\documentclass[11pt]{amsart}
\pdfoutput=1
\usepackage{latexsym, graphicx, epsfig, amsmath, amsfonts,amssymb}
\usepackage{multirow}
 \usepackage{diagbox}
\usepackage{color}
 \usepackage[percent]{overpic}
\usepackage[ruled,vlined,linesnumbered]{algorithm2e}
\usepackage{algpseudocode}
\usepackage{threeparttable}
\usepackage[T1]{fontenc}
\usepackage{booktabs}
\usepackage{hyperref}
\usepackage{caption}
\usepackage{subcaption}  
\usepackage{amsmath}
\usepackage{extarrows}
\usepackage{booktabs}     
\usepackage{siunitx}      
\usepackage{caption}  
\usepackage{amssymb}
\usepackage{epstopdf,mathtools,bbm}
\usepackage[left=1.25in,right=1.25in,top=1in]{geometry}

\usepackage{tikz}
\usetikzlibrary{calc}
\usepackage[percent]{overpic}
\usepackage{amsthm} 
\usepackage{mathabx}
\usetikzlibrary{arrows.meta}

\usepackage{mathrsfs}
\newcounter{Rownumber} 
\usepackage{multirow}
\usepackage{tabularx}
\usepackage{multirow}
\usepackage{makecell}
\usepackage{booktabs}
\usepackage{multirow}
\usepackage[font=tiny,labelfont=bf]{caption}
\usepackage{booktabs}
\usepackage{multirow}
\usepackage{mathrsfs}
\usetikzlibrary{arrows.meta,calc,decorations.pathmorphing,decorations.markings}
\usepackage{hyperref}
\newenvironment{keyword}
{\par\vspace{0.5em}\noindent\textbf{Keywords:}\ }
{\par\vspace{0.5em}}

\newtheorem{theorem}{Theorem}[section]

\newtheorem{corollary}{Corollary}[section]

\newtheorem{remark}{Remark}

\makeatletter

\newtheoremstyle{condstyle}%
  {\abovedisplayskip}{\belowdisplayskip}
  {}
  {}
  {\bfseries}
  {}
  {0.5em}
  {\thmname{#1}\ \thmnumber{#2.}\thmnote{\ \textnormal{(#3)}}}

\theoremstyle{condstyle}

\theoremstyle{plain} 
\title{Inverse  scattering for  three-dimensional random obstacles with multi-frequency data}

\author[Z. Sun]{Zhiqi Sun$^{1}$}
\author[Y. Lin]{Yiwen Lin$^{2}$}

\begin{document}
\graphicspath{figures/}
\maketitle

\begingroup
\renewcommand{\thefootnote}{\arabic{footnote}}
\footnotetext[1]{School of Mathematical Sciences, Shanghai Jiao Tong University,
Shanghai 200240, People's Republic of China. Email: \texttt{sunzhq1016@sjtu.edu.cn}.}
\footnotetext[2]{School of Mathematical Sciences, Shanghai Jiao Tong University,
Shanghai 200240, People's Republic of China. Email: \texttt{linyiwen@sjtu.edu.cn}.}
\endgroup

\begin{abstract}

In many practical scenarios the shapes of scatterers exhibit uncertain geometric variations arising from diverse physical or environmental factors. For inverse scattering problems which are inherently ill-posed, the presence of such geometric uncertainties may have a non-negligible impact on the recovery process. With the aim of recovering both  obstacle geometry and  statistics of the shape uncertainties, in this paper we study an inverse acoustic scattering problem for three-dimensional smooth star-shaped obstacles with random isotropic fluctuations. We propose an efficient Monte Carlo-based multi-frequency recursive linearization algorithm in which the far-field operator is linearized with respect to the geometry parameters and frequency continuation is employed to  recover the unknown geometry from coarse to fine scales. Based on the reconstructed samples, we further estimate the reference
geometry and key statistics of the shape fluctuation field including  Karhunen--Lo\`eve eigenvalues, covariance hyper-parameters for
Gaussian perturbations and covariance structure, representative marginal
distributions for  non-Gaussian perturbations. We also prove that the probability law of the  far-field data uniquely determines the   radial function in distribution which implies uniqueness of the reference shape and related statistics. Numerical experiments demonstrate the effectiveness of the proposed method in recovering both the scatterer shapes and the associated statistical information under Gaussian and non-Gaussian random  variations.

\end{abstract}


\pagestyle{myheadings}
\thispagestyle{plain}
\markboth{}
{ Inverse Scattering Problem of Three-dimensional Random  Obstacle}

\section{Introduction}

Inverse obstacle scattering problems have important applications in fields such as geophysical exploration \cite{colton1998inverse}, optical design \cite{Bao2001MathematicalModeling} and medical imaging \cite{Kuchment2014}. In these scenarios, the inverse problem typically aims to recover the shape of an unknown obstacle from measured scattering data. In practice, however, the geometry of the real obstacle is often affected by factors such as manufacturing tolerances, machining deviations, surface contamination, corrosion and wear, deposit accumulation, material fatigue and environmental loads \cite{wang2020review}. Such effects induce random departures of the geometry from its ideal shape, resulting in uncertain, non-negligible perturbations and  consequent influence on the shape recovery which  is intrinsically ill-posed. A quantitative characterization of these random fluctuations provides useful information on the variability of the scatterer and is important for assessing the reliability of the reconstructed geometry. Therefore, beyond reconstructing the obstacle geometry itself, it is desirable to further estimate the statistical features of the underlying shape fluctuation field from the measured data.  In general, such uncertain shape variations are typically small relative to the size of the scatterer itself: rather than modifying the overall geometry, they manifest primarily as local fluctuations or slight deviations from its original shape. Based on this observation, this paper does not address arbitrary large-scale random deformations. Instead, we focus on the reconstruction of three-dimensional scatterers with small  shape perturbations, together with the recovery of statistical parameters associated with the underlying fluctuation field.

  Research on inverse problems with uncertainties have received increasing attention in recent years, related studies include inverse random source problems \cite{MR3120587,MR3565588,MR4653392} and inverse random potential problems \cite{MR2386724,MR4712402,MR4918624}, etc. For inverse random shape problems, existing works have mainly focused on the reconstruction of one-dimensional  gratings; see, for example, \cite{MR4164073,MR4693214,MR5072888}. More recently, inverse scattering problems for two-dimensional random obstacles have  been investigated in \cite{sun2026inverse}.  For the recovery of three-dimensional  obstacle with uncertainties,  limited studies are available.   D\"olz et al. \cite{MR4337758} investigated this problem within a Bayesian framework based on a random deformation model and reported the posterior expected shape and variance of the random scatterer in the single-frequency setting. A more direct characterization of the statistical structure of the underlying geometric fluctuations from inverse measurements, however, was less emphasized in previous study.

To fill this gap, this work studies inverse acoustic scattering by three-dimensional star-shaped scatterers with small-amplitude smooth random perturbations, aiming to recover not only the obstacle geometry but also  statistics of the underlying shape uncertainties.
 The main geometry of the scatterer  is represented by a deterministic reference shape together with
a random boundary fluctuation modeled by an isotropic random field defined on the unit
sphere $\mathbb{S}^2$.  To achieve our goals, we propose an efficient Monte Carlo-based multi-frequency recursive linearization algorithm for the shape recovery: the key idea is to reformulate the uncertain obstacle recovery into a collection of deterministic inversion problems with individual realizations. This strategy connects the present formulation with deterministic  shape reconstruction, for which representative methods include \cite{MR883771,MR1181580,MR2363787,MR4929108}. For each  realization, we linearize the far-field operator around the current approximate geometry and determine the shape correction  via the Gauss–Newton method. A multi-frequency continuation strategy is then employed where low-frequency data recover the overall outline of the scatterer and higher-frequency data are gradually incorporated to resolve finer details. Based on the Colton--Sleeman uniqueness theorem \cite{MR729385}, we show that within our  random scattering setting, the distribution of the random radial function is uniquely determined by the probability law of the far-field data. Under the zero-mean perturbation assumption, this implies that the reference shape and the covariance structure of the associated random perturbation field are uniquely determined.
The numerical results demonstrate the effectiveness of the proposed method for both shape reconstruction and statistical recovery: for Gaussian perturbations, the dominant Karhunen--Lo\`eve (KL) eigenvalues and the covariance hyper-parameters, including the perturbation amplitude $\sigma$ and correlation length $\ell_c$, are quantitatively estimated; for non-Gaussian perturbations, the reconstructed samples capture representative marginal distributions and spatial covariance structures. These results indicate that the proposed method is capable of recovering statistical features of random boundaries beyond deterministic shape reconstruction. In particular, the comparisons with single-frequency inversions show that the multi-frequency continuation strategy provides a more reliable balance between stability and resolution, which is important for the subsequent recovery of statistical quantities from the reconstructed samples. To the best of our knowledge, this work is among the first attempts to recover both three-dimensional random obstacle shapes and the associated statistical structure of the underlying random field from far-field measurements. The main contributions of this work are summarized as follows:

\begin{itemize}
\item A unified and tractable radial model of three-dimensional random star-shaped obstacles is formulated for the inverse acoustic scattering problem, providing a flexible framework to represent and quantify shape uncertainties arising from Gaussian and non-Gaussian perturbations. 
\item We propose a Monte Carlo-based multi-frequency recursive linearization method for reconstructing three-dimensional random scatterers, integrating realization-wise inversion with frequency continuation to improve reconstruction stability and resolution while enabling reliable statistical estimation.

\item Beyond shape reconstruction, the reconstructed samples enable quantitative estimation of the statistical quantities characterizing the underlying perturbation field.

\item We prove a statistical uniqueness result showing that the probability law of the random far-field data determines the stochastic radial function in distribution, and hence uniquely determines its mean and covariance, with the deterministic reference shape identified under the zero-mean perturbation assumption.

 \end{itemize}

The rest of this paper is organized as follows. In \autoref{Problem setup and preliminaries}, we introduce the problem setting and formulate the random star-shaped obstacle model. In \autoref{Reconstruction method}, we present the proposed reconstruction method, including the forward solver, the multi-frequency recursive linearization algorithm, the recovery of statistical quantities and a statistical uniqueness result. \autoref{Numerical_experiments} presents four classes of numerical experiments for both Gaussian and non-Gaussian perturbations to demonstrate the effectiveness of the proposed method. Finally, \autoref{Conclusion} concludes the paper with a summary of the main findings and possible future directions.

\section{Problem setup}\label{Problem setup and preliminaries}

\subsection{Random scattering problem}
Consider a bounded impenetrable scatterer in $\mathbb{R}^3$. In many practical scenarios, the shape of the scatterer often exhibits unavoidable
geometric uncertainty due to physical or environmental factors, hence it is  more appropriate to model the scatterer as a random obstacle rather than as a fixed deterministic object. Mathematically, let $(\Omega,\mathscr F,\mathbb P)$ be a probability space and let $D(\omega)\subset\mathbb R^3$ denote the realization of the random scatterer with $\omega\in\Omega$. We assume that for $\mathbb P$-almost every  $\omega$, $D(\omega)$ is bounded and has a sufficiently smooth boundary $\partial D(\omega)$. Given a time-harmonic incident acoustic plane wave field $u^i(x,d)=e^{ikx\cdot d}$ where $d\in\mathbb{S}^2$ is an incident direction and $k>0$ is the wavenumber, the forward scattering problem for the realization $D(\omega)$ is to determine the scattered field $u^s(x,\omega;d,k)$. For each fixed $\omega$, the scattered field satisfies the Helmholtz equation
\begin{equation}\label{Helmholtz}
    \Delta u^s(x,\omega;d,k)+k^2u^s(x,\omega;d,k)=0,\quad x\in\mathbb{R}^3\backslash\widebar{D(\omega)}.
\end{equation}
The total field $u(x,\omega;d,k)=u^i(x,d)+u^s(x,\omega;d,k)$ is subject to the Dirichlet boundary condition:
\begin{equation}\label{Dirichlet}
    u(x,\omega;d,k)=u^i(x,d)+u^s(x,\omega;d,k)=0 \quad \text{on }\partial D(\omega).
\end{equation}
To ensure the outgoing property of $u^s$, the Sommerfeld radiation condition is introduced:
\begin{equation}\label{sommer}
    \lim_{r\rightarrow\infty}r\bigg(\frac{\partial u^s}{\partial r}(x,\omega;d,k)-iku^s(x,\omega;d,k)\bigg)=0, \quad r=|x|.
\end{equation}
Under the radiation condition \eqref{sommer}, the scattered field $u^s$ possesses the asymptotic behavior
\begin{equation}\label{far}
    u^s(x,\omega;d,k)=\frac{\mathrm{e}^{\mathrm{i} k r}}{r}\left\{u^{\infty}(\widehat{x},\omega;d,k)+\mathcal{O}\left(\frac{1}{r}\right)\right\},\quad r=|x| \rightarrow \infty,
\end{equation}
where $\widehat{x}=\frac{x}{|x|}\in\mathbb{S}^2=\{x\in\mathbb{R}^3:|x|=1\}$ is the observation direction and the far-field pattern $u^\infty(\widehat{x},\omega;d,k)$ is an analytic function of $\widehat{x}$. In the present setting, for $\mathbb P$-almost every
$\omega\in\Omega$, the forward scattering problem \eqref{Helmholtz}-\eqref{sommer} for the smooth, bounded and
impenetrable obstacle $D(\omega)$ is well posed  \cite{colton1998inverse,MR4142771}. Hence, the corresponding far-field pattern
$u^\infty(\cdot,\omega;d,k)$ is well defined. \autoref{fig:scattering_schematic} presents the basic scattering configuration in $\mathbb{R}^3$ and  \autoref{ranex} illustrates several realizations of a randomly perturbed scatterer.

\begin{figure}
\centering
\resizebox{0.2\textwidth}{!}{%
\begin{tikzpicture}[line cap=round, line join=round, >=Latex]

\def\A{3.2}   
\def\B{4.1}   

\tikzset{
    meshline/.style={blue!80!black, line width=0.55pt},
    boundary/.style={blue!80!black, line width=1.0pt},
    scatter/.style={->, blue!80!black, line width=1.4pt},
    incident/.style={->, orange!90!black, line width=1.8pt},
    incidentfront/.style={orange!90!black, line width=1.3pt},
    labb/.style={blue!80!black},
    labo/.style={orange!90!black}
}

\draw[boundary] (0,0) ellipse[x radius=\A, y radius=\B];

\begin{scope}
    \clip (0,0) ellipse[x radius=\A, y radius=\B];

    \foreach \s in {0.12,0.24,0.36,0.48,0.60,0.72,0.84,0.92}{
        \draw[meshline] (0,0) ellipse[x radius={\A*\s}, y radius=\B];
    }

    \foreach \t in {-0.92,-0.82,-0.72,-0.62,-0.52,-0.42,-0.32,-0.22,-0.12,-0.02,
                     0.08,0.18,0.28,0.38,0.48,0.58,0.68,0.78,0.88}{
        \pgfmathsetmacro{\xr}{\A*sqrt(max(0,1-\t*\t))}
        \pgfmathsetmacro{\yr}{0.17*\A*(1-\t*\t)}
        \draw[meshline] (0,\B*\t) ellipse[x radius=\xr, y radius=\yr];
    }
\end{scope}

\node at (1.1,-2.05) {\fontsize{32}{32}\selectfont $D$};


\begin{scope}[shift={(-6.36,-0.13)}, rotate=0]
    \draw[incident] (0,0) -- (2.76,0);

    \draw[incidentfront] (0.35,-0.85) -- (0.35,0.85);
    \draw[incidentfront] (0.78,-0.62) -- (0.78,0.62);
    \draw[incidentfront] (1.18,-0.45) -- (1.18,0.45);

    \node[labo] at (1.35,0.92) {\fontsize{30}{30}\selectfont $u^i$};
\end{scope}

\begin{scope}[shift={(5.60,3.33)}, rotate=-155]
    \draw[incident] (0,0) -- (2.75,0);

    \draw[incidentfront] (0.38,-0.82) -- (0.38,0.82);
    \draw[incidentfront] (0.82,-0.60) -- (0.82,0.60);
    \draw[incidentfront] (1.22,-0.42) -- (1.22,0.42);

    \node[labo] at (1.55,0.98) {\fontsize{30}{30}\selectfont $u^i$};
\end{scope}

\draw[scatter]
    (-2.75,2.55)
    .. controls (-3.10,2.90) and (-3.75,3.00) ..
    (-4.20,3.15)
    .. controls (-4.80,3.35) and (-5.05,4.08) ..
    (-5.70,4.25)
    node[pos=0.80, above left, xshift=-2pt, yshift=2pt, labb]
    {\fontsize{26}{26}\selectfont $u^s$};
    
\draw[scatter]
    (4.10,-1.55)
    .. controls (4.55,-1.62) and (5.05,-1.88) ..
    (5.30,-2.30)
    .. controls (5.55,-2.72) and (6.15,-2.68) ..
    (6.85,-2.95)
    node[pos=0.63, below, xshift=6pt, yshift=-6pt, labb]
    {\fontsize{26}{26}\selectfont $u^s$};
\draw[scatter]
    (-1.55,-4.10)
    .. controls (-1.70,-5.05) and (-2.35,-5.55) ..
    (-2.55,-6.75)
    node[pos=0.72, right, xshift=6pt, yshift=-2pt, labb]
    {\fontsize{26}{26}\selectfont $u^s$};




\end{tikzpicture}%
}
\caption{\scriptsize  Schematic illustration of the incident and scattered fields.}
\label{fig:scattering_schematic}
\end{figure}
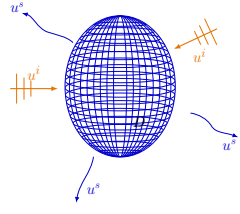

For the inverse scattering problem considered in this paper, we aim to recover the reference geometry of the scatterer $D$ and to quantify the statistical features of the  random shape fluctuations from the corresponding far-field patterns $u^\infty(\widehat{x},\omega;d,k)$. It is well-known that inverse shape reconstruction from far-field data is nonlinear
and ill-posed in the sense of Hadamard \cite{Hadamard1923}. When the
scatterer geometry is subject to random fluctuations, this intrinsic ill-posedness is
further compounded by the variability of the obstacle shape, making stable and
efficient reconstruction even more challenging.

\begin{figure}
\begin{center}
    \begin{overpic}[width=1\textwidth, tics=10]{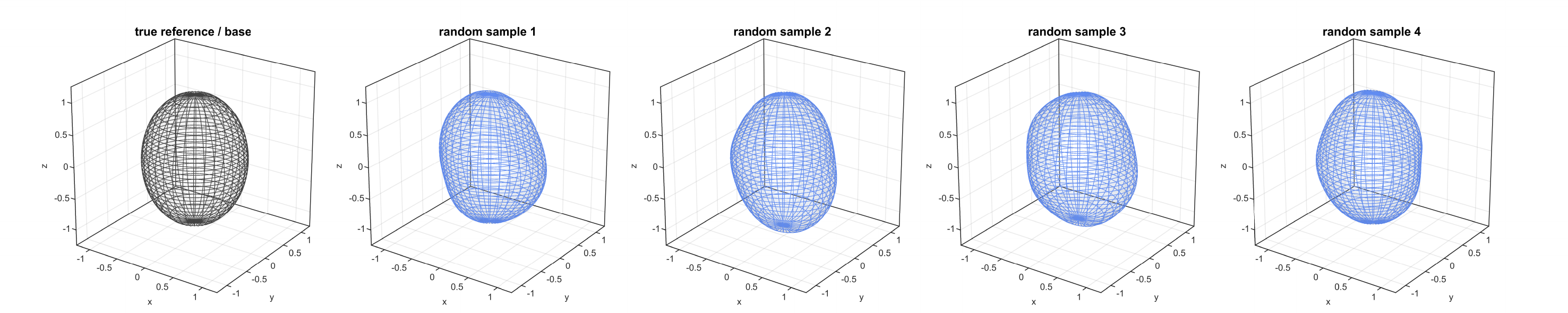}
    \end{overpic}
\end{center}
\caption{\scriptsize  Random realizations of the scatterer under Gaussian random field}
\label{ranex}
\end{figure}

The different realizations of a random scatterer motivate the need for a tractable and physically meaningful description of the random  shape variations. To this end, we introduce below a stochastic  model of three-dimensional star-shaped obstacles in which the random  perturbation will be modeled as a random field on $\mathbb{S}^2$.

\subsection{Random obstacle model and formulation}
\label{randomob}
To obtain a tractable representation of the uncertain geometry, let
$D(\omega)\subset\mathbb R^3$ denote the random obstacle with boundary
$\Gamma(\omega)=\partial D(\omega)$, where $\omega\in\Omega$ indexes
geometric realizations of the scatterer. For each fixed $\omega$, $D(\omega)$ is a
deterministic obstacle realization. We assume that, for $\mathbb P$-almost every $\omega\in\Omega$, the obstacle
$D(\omega)$ is star-shaped with respect to the  fixed point
$x_0\in\mathbb R^3$. Therefore, the boundary of each realization
can be represented by a positive and sufficiently smooth radial function
$r(\omega,\cdot):\mathbb S^2\to\mathbb R^+$. More precisely,
\begin{equation}\label{radius}
   \partial D(\omega)=\Gamma(\omega)
   =
   \{\,x_0+r(\omega,\hat{x})\hat{x}:\hat{x}\in\mathbb{S}^2\,\},
\end{equation}
where $\hat{x}\in\mathbb S^2$ denotes a direction on the unit sphere.
To separate the deterministic reference geometry from the random
variation, we decompose the radial function as
\begin{equation}
    r(\omega,\hat{x})
    =
    \bar{r}(\hat{x})
    +
    \rho(\omega,\hat{x}),
\end{equation}
here $\bar{r}$ is the deterministic reference radial function and
$\rho\in L^2\bigl(\Omega;L^2(\mathbb{S}^2)\bigr)$ is a zero-mean random
perturbation satisfying
\begin{equation}\label{exp}
    \mathbb{E}\bigl[\rho(\cdot,\hat{x})\bigr]=0,
    \qquad
    \hat{x}\in\mathbb{S}^2.
\end{equation}
The covariance structure of the perturbation field
$\rho(\omega,\hat{x})$ is taken to be isotropic on $\mathbb{S}^2$,
reflecting that no specific directional preference is imposed on  geometric fluctuations and  the correlations are assumed to depend only on the angular
separation. Since $\rho$ satisfies \eqref{exp}, its second-order dependence is
characterized by the covariance function
\begin{equation}
C_\rho(\hat{x},\hat{y})
=
\operatorname{Cov}
\bigl(\rho(\omega,\hat{x}),\rho(\omega,\hat{y})\bigr).
\end{equation}
This isotropy assumption means that the covariance is invariant under rotations.
More precisely,  for any  $Q\in SO(3)$, 
\begin{equation}
C_\rho(Q\hat{x},Q\hat{y})
=
C_\rho(\hat{x},\hat{y}),
\qquad
Q\in SO(3),
\end{equation}
where $SO(3)$ denotes the group of three-dimensional rotation matrices.
 Consequently, $C_\rho(\hat{x},\hat{y})$ depends only on the angular separation between
$\hat{x}$ and $\hat{y}$. Thus, for any $\hat{x},\hat{y}\in\mathbb{S}^2$,
\begin{equation}\label{iso}
\operatorname{Cov}
\bigl(\rho(\omega,\hat{x}),\rho(\omega,\hat{y})\bigr)
=
C_\rho(\hat{x},\hat{y})
=
C_\rho(\hat{x}\cdot\hat{y}).
\end{equation}
Since $\hat{x}\cdot\hat{y}=\cos\gamma$ where  $\gamma$ denotes the angular separation between $\hat{x}$ and
$\hat{y}$,  $C_\rho$ may equivalently be regarded as a function of
$\gamma$.

The isotropic covariance function $C_\rho$ naturally induces a covariance operator $\mathcal C_\rho$
on $L^2(\mathbb S^2)$. The spectral decomposition of $\mathcal C_\rho$ provides the well-known 
Karhunen--Lo\`eve (KL) representation of $\rho$ in the mean-square sense, where the
eigenvalues describe the strength of the fluctuation modes and the corresponding
KL coefficients are centered and mutually uncorrelated. In the Gaussian case,
these coefficients are also independent, so the  mean and covariance completely
determine the probability law of the perturbation field. For non-Gaussian
perturbations, however, the covariance characterizes only the second-order
dependence structure. Therefore, additional statistical quantities are
needed to further characterize the distributional properties of the perturbation field.

For the Gaussian random field, the general KL representation  above can be made more explicit
by using the isotropic structure on $\mathbb S^2$. In this case, the
covariance operator $\mathcal C_\rho$ is diagonalized by the real spherical
harmonics on $\mathbb{S}^2$ \cite{MR3404631,marinucci2011random}, and the perturbation field $\rho$ admits the KL expansion
\begin{equation}\label{Crho}
\rho(\omega,\hat{x})
=
\sum_{\ell=0}^{\infty}
\sum_{m=-\ell}^{\ell}
\sqrt{\lambda_\ell}\,
\xi_{\ell m}(\omega)
Y_{\ell m}(\hat{x}),
\end{equation}
where $\ell$ and $m$ denote the degree and order of the real spherical harmonics,
respectively, $\lambda_\ell$ is the degree-wise KL eigenvalue and
$\{\xi_{\ell m}\}$ are mutually independent standard normal random variables.
Equivalently,
\begin{equation}
\xi_{\ell m}\sim\mathcal{N}(0,1),
\qquad
\mathbb{E}[\xi_{\ell m}\xi_{\ell' m'}]
=
\delta_{\ell\ell'}\delta_{mm'}.
\end{equation}
The covariance operator
$\mathcal{C}_\rho:L^2(\mathbb{S}^2)\to L^2(\mathbb{S}^2)$ is defined by
\begin{equation}
(\mathcal{C}_\rho f)(\hat{x})
=
\int_{\mathbb{S}^2}
C_\rho(\hat{x},\hat{y})f(\hat{y})\,dS(\hat{y}).
\end{equation}
Its eigenpairs satisfy
\begin{equation}
\mathcal{C}_\rho Y_{\ell m}
=\int_{\mathbb{S}^2}
C_\rho(\hat{x},\hat{y})Y_{\ell m}(\hat{y})\,dS(\hat{y})=\lambda_\ell Y_{\ell m}(\hat{x}).
\end{equation}
For an isotropic covariance kernel, each eigenvalue $\lambda_\ell$  independent of $m$  has multiplicity $2\ell+1$. 
For later numerical implementation, the expansion \eqref{Crho} is truncated at degree
$L_{\max}$:
\begin{equation}\label{KL}
\rho(\omega,\hat{x})
\approx
\sum_{\ell=0}^{L_{\max}}
\sum_{m=-\ell}^{\ell}
\sqrt{\lambda_\ell}\,
\xi_{\ell m}(\omega)
Y_{\ell m}(\hat{x}).
\end{equation}
The accuracy of this finite-dimensional approximation is governed by
the decay of the covariance spectrum. In particular, a faster decay of
$\{\lambda_\ell\}$ leads to a smaller truncation error. To generate smooth and spatially correlated perturbations,  the classical 
squared-exponential covariance kernel is adopted here:
\begin{equation}\label{cov}
C_\rho(\hat{x},\hat{y})
=
\sigma^2
\exp\left(
-\frac{\|\hat{x}-\hat{y}\|^2}{\ell_c^2}
\right),
\qquad
\hat{x},\hat{y}\in\mathbb{S}^2,
\end{equation}
where $\|\cdot\|=\|\cdot\|_{\mathbb{R}^3}$ denotes the Euclidean norm
in $\mathbb{R}^3$.  Since
\begin{equation}
\hat{x}\cdot\hat{y}=\cos\gamma,
\qquad
\|\hat{x}-\hat{y}\|^2
=
2-2\hat{x}\cdot\hat{y}
=
4\sin^2\left(\frac{\gamma}{2}\right),
\end{equation}
the covariance kernel in \eqref{cov} can equivalently be written as
\begin{equation}
C_\rho(\hat{x},\hat{y})
=
\sigma^2
\exp\left(
-\frac{2(1-\hat{x}\cdot\hat{y})}{\ell_c^2}
\right)
=
\sigma^2
\exp\left(
-\frac{4\sin^2(\gamma/2)}{\ell_c^2}
\right).
\end{equation}
Hence, the covariance depends only on the angular separation and is
consistent with the isotropic covariance assumption.  The parameter $\sigma^2$ in \eqref{cov} is the pointwise variance of the perturbation
field since $C_\rho(\hat{x},\hat{x})=\sigma^2.$ Thus, $\sigma$ controls the fluctuation amplitude, while $\ell_c$
determines the characteristic spatial correlation scale. Larger values
of $\ell_c$ produce more slowly varying perturbations, whereas smaller
values lead to more localized fluctuations. 

Note that the squared-exponential kernel in \eqref{cov} is shown to be  continuous, symmetric and positive
semi-definite \cite{oksendal2013stochastic}, therefore \eqref{cov} defines a
compact, self-adjoint and positive covariance operator on
$L^2(\mathbb{S}^2)$. The eigenvalues $\{\lambda_\ell\}$ in
\eqref{KL} are the eigenvalues of $\mathcal{C}_\rho$ associated with
the covariance kernel \eqref{cov}. In this sense, within the chosen squared-exponential covariance model, the pair
$(\sigma,\ell_c)$ provides a concise and complete parametrization of the fluctuation
strength and spatial correlation of the Gaussian perturbation field.

\begin{remark}
Since spherical coordinates are not globally regular on $\mathbb S^2$ and the
same azimuthal difference may correspond to different physical distances at
different latitudes, the present three-dimensional stochastic model, under the
isotropy assumption, characterizes correlations through the angular separation
between directions on $\mathbb S^2$  rather than prescribing them directly through
separations between boundary parameters as in the two-dimensional perturbation
model \cite{sun2026inverse}.  Accordingly, \eqref{iso} provides a geometrically intrinsic and computationally convenient measure for
characterizing  correlations of the isotropic random field in the three-dimensional setting.

\end{remark}

\begin{remark}
     In our present context the squared-exponential kernel \eqref{cov} is adopted as a representative isotropic covariance model since it produces smooth shape fluctuations and provides a classical two-parameter description of the perturbation amplitude and correlation length. The random field formulation is not restricted to this particular choice and can accommodate other covariance models, provided that they define valid covariance operators on $L^2(\mathbb S^2)$ and generate perturbations with sufficient regularity for stable reconstruction. Moreover, since the random obstacle parametrization in \eqref{radius}--\eqref{exp}
is  independent of the governing wave equation, the same
geometrical description can also be adapted to other settings such as elastic or electromagnetic scattering
problems. In this work, we focus mainly on the acoustic scattering setting to develop and
demonstrate the proposed reconstruction method.
\end{remark}

\section{Reconstruction method for random obstacles}
\label{Reconstruction method}

For random scatterers with isotropic shape fluctuations, the inverse problem considered here seeks to recover the reference shape and characterize the statistical structure of the associated perturbation field. In three dimensions, solving this inverse problem is particularly
challenging because the far-field data depend nonlinearly on the
obstacle geometry and the inversion is ill-posed. This
difficulty is further compounded by the need to accurately resolve
geometric variations across different realizations for reliable
statistical estimation. Multi-frequency measurements are known to improve stability in time-harmonic scattering problems \cite{bao2015inverse}, and frequency continuation allows the reconstruction to proceed progressively from coarse to fine scales. Inspired by \cite{MR2363787,MR4162000,MR2998714}, we propose an efficient Monte Carlo-based multi-frequency recursive linearization algorithm for the  inverse random obstacle problem. Since this procedure requires repeated forward solves over multiple realizations
and frequencies, we first introduce the forward solver used in the later
multi-frequency reconstruction algorithm.

\subsection{Forward problem}\label{sub2}

 For a fixed realization $\omega$, the random radial function $r(\omega,\hat x)$ becomes deterministic and so does the corresponding scatterer. Let $u^i=e^{ikx\cdot d}$ be the incident plane wave and let $D$ be the given scatterer with boundary $\Gamma=\partial D$. By Green's representation formula, the scattered field admits the representation
\begin{equation}\label{Helmholtzre}
    u^s(x)
=
\int_\Gamma
\left[
u^s(y)\frac{\partial \Phi(x,y)}{\partial \nu(y)}
-
\Phi(x,y)\frac{\partial u^s}{\partial \nu}(y)
\right]d\sigma_y ,
\qquad x\in \mathbb R^3\setminus \widebar{D},
\end{equation}
where $\Phi(x,y)$ is the outgoing fundamental solution of the Helmholtz equation in $\mathbb R^3$ and $\nu$ denotes the unit outward normal to $\Gamma$. Using the sound-soft boundary condition \eqref{Dirichlet}, the exterior total field can be represented in terms of the Neumann trace $\partial u/\partial \nu$ as
\begin{equation}\label{31}
    u(x)
+
\int_\Gamma
\frac{\partial u}{\partial \nu}(y)\Phi(x,y)d\sigma_y
=
u^i(x),
\qquad x\in\mathbb R^3\setminus\widebar{D}.
\end{equation}
Accordingly, once the Neumann data $\partial u/\partial\nu$ on $\Gamma$ are determined, the far-field pattern is obtained from
\begin{equation}
    u^\infty(\widehat x,d)
=
-\frac{1}{4\pi}
\int_\Gamma
e^{-ik \widehat x\cdot y}
\frac{\partial u}{\partial \nu}(y,d)d\sigma_y,
\qquad \widehat x,d\in \mathbb{S}^2.
\end{equation}
To determine $\partial u/\partial \nu$, we introduce the single-layer boundary operator $S$ and its normal derivative operator $K'$, defined by
\[
(S\rho)(x)=\int_{\Gamma}\Phi(x,y)\rho(y)\,d\sigma_y,
\qquad
(K'\rho)(x)=\int_{\Gamma}
\frac{\partial \Phi(x,y)}{\partial \nu(x)}\rho(y)\,d\sigma_y ,
\qquad x\in\Gamma .
\]
Taking the exterior normal derivative of \eqref{31} and applying the jump relation for the normal derivative of the single-layer potential yields
\begin{equation}\label{neu1}
    \left(\frac12 I+K'\right)\frac{\partial u}{\partial\nu}
=
\frac{\partial u^i}{\partial\nu}
\qquad \text{on }\Gamma .
\end{equation}
The boundary trace of \eqref{31} gives
\begin{equation}\label{neu2}
    S\frac{\partial u}{\partial\nu}
=
u^i
\qquad \text{on }\Gamma .
\end{equation}
A combined-field equation is then obtained by multiplying \eqref{neu2} by $-i\eta_1$ and adding it to \eqref{neu1}:
\begin{equation}\label{neumann}
    \left(
\frac12 I+K'-i\eta_1 S
\right)
\frac{\partial u}{\partial \nu}
=
\frac{\partial u^i}{\partial \nu}
-
i\eta_1 u^i
\qquad \text{on } \Gamma .
\end{equation}
Here $\eta_1>0$ is a coupling parameter, typically chosen as $\eta_1=k/2$. The combined-field term $-i\eta_1 S$ is introduced to ensure the unique solvability of the boundary integral equation and to improve numerical stability. After solving the boundary integral equation \eqref{neumann} for $\partial u/\partial\nu$, the forward far-field data are evaluated by the preceding far-field representation.

\subsection{The reconstruction of mean shape}\label{sub3}
Denote by $r_j(\hat x):=r(\omega_j,\hat x)$ the true radius corresponding to the 
$j$-th sample and by $u_{j}^{\infty,\delta}(\widehat x,d;k)$ its measured far‑field data with noise $\delta$, here $\widehat x\in\mathbb S^2$ is the observation direction, $d\in\mathbb S^2$ the incident direction and $k$ the wavenumber. For a given sample $\omega_j$, define the nonlinear far-field operator $\mathcal{F}_k:\partial D\rightarrow{L}^2(\mathbb{S}^2)$ from the radial function to the far‑field data as 
\begin{equation}
   \mathcal F_k(r_{j})(\widehat x,d)=u_{j}^\infty(\widehat x,d;k).
\end{equation}
To address this nonlinear inverse problem, we use the domain differentiability of
the far-field operator under boundary deformations \cite{MR1203018} and employ a  multi-frequency recursive linearization strategy.   Let $r_{j,n}$ be the current approximation to $r_{j}$ at the $n$-th iteration.  Define the update as
\begin{equation}\label{h}
    h_{j,n}(\hat x)
    =
    r_{j,n+1}(\hat x)-r_{j,n}(\hat x).
\end{equation}
Linearizing the far‑field operator $\mathcal F_k$ around $r_{j,n}$ using the Fr\'echet derivative $\mathcal F_k'[r_{j,n}]$  gives
\begin{equation}\label{RLA}
     \mathcal F_k(r_{j,n+1})
     =
     \mathcal F_k(r_{j,n}+h_{j,n})
     \approx
     \mathcal F_k(r_{j,n})
     +
     \mathcal F_k'[r_{j,n}]h_{j,n}.
\end{equation}
Suppose the current boundary is $\Gamma_{j,n}=\{r_{j,n}(\hat x)\hat x:\hat x\in\mathbb S^2\}$, then the  displacement induced by the radius increment  $h_{j,n}$ is 
\begin{equation}
    V_{j,n}(\hat x)
    =
    h_{j,n}(\hat x)\hat x.
\end{equation}
Projecting this displacement onto the outward normal direction gives
\begin{equation}
    V_{j,n}\cdot \nu_{j,n}
    =
    h_{j,n}(\hat x)(\hat x\cdot \nu_{j,n}),
\end{equation}
where  $\nu_{j,n}$  is the unit outward normal on the current boundary $\Gamma_{j,n}$. Let $v_{j,n}$ denote the shape derivative field associated with the perturbation $V_{j,n}$. Then $v_{j,n}$ satisfies 
\begin{equation}\label{vn}
\begin{gathered}
\Delta v_{j,n}+k^2v_{j,n} = 0
\quad \text{in } \mathbb R^3\setminus \overline{D_{j,n}}, \\[1mm]
v_{j,n}
=
-h_{j,n}(\hat x)(\hat x\cdot \nu_{j,n})
\frac{\partial u_{j,n}}{\partial \nu_{j,n}}
\quad \text{on } \Gamma_{j,n}, \\[1mm]
\lim_{r\to\infty} r\left(
\frac{\partial v_{j,n}}{\partial r}-ikv_{j,n}
\right)=0,
\quad r=|x|.
\end{gathered}
\end{equation}
Here  $u_{j,n}$ is the total field corresponding to the current shape $r_{j,n}$ and $\partial u_{j,n}/\partial \nu_{j,n}$ denotes its normal derivative on  $\Gamma_{j,n}$. The far‑field pattern $v_{j,n}^\infty$ of $v_{j,n}$ is precisely the first‑order response of the far‑field operator to the radius increment $h_{j,n}$, i.e.
\begin{equation}\label{deria}
    \mathcal F_k'[r_{j,n}]h_{j,n}
    =
    v_{j,n}^\infty(\cdot,\cdot;k).
\end{equation}
Substituting \eqref{deria} into \eqref{RLA}, we obtain the following linear approximation
of the far-field data near the current shape:
\begin{equation}\label{linear}
    \mathcal F_k(r_{j,n}+h_{j,n})
    \approx
    \mathcal F_k(r_{j,n})
    +
    \mathcal F_k'[r_{j,n}]h_{j,n}
    =
    \mathcal F_k(r_{j,n})
    +
    v_{j,n}^\infty .
\end{equation}
To solve for $v_{j,n}$, suppose that $v_{j,n}$ admits a combined-density representation
with unknown density $\varphi_{j,n}$:
\begin{equation}\label{combine}
    v_{j,n}(x)
    =
    \int_{\Gamma_{j,n}}
    \left(
    \frac{\partial \Phi(x,y)}{\partial \nu(y)}
    -
    i\eta_2 \Phi(x,y)
    \right)
    \varphi_{j,n}(y)\,d\sigma_y,
    \qquad x\in \mathbb R^3\setminus \overline{D_{j,n}},
\end{equation}
where $\eta_2$ is the coupling parameter. Then substituting \eqref{combine} into \eqref{vn} and using the jump relations, we obtain the
boundary integral equation
\begin{equation}\label{bie}
\left(
\frac12 I + K_{j,n} - i\eta_2 S_{j,n}
\right)\varphi_{j,n}
=
-(V_{j,n}\cdot\nu_{j,n})
\frac{\partial u_{j,n}}{\partial \nu_{j,n}},
\qquad \text{on } \Gamma_{j,n}.
\end{equation}
Once $\varphi_{j,n}$ is determined, the far-field pattern $v_{j,n}^\infty$ can be
evaluated directly.

At a fixed frequency $k$, the update at the $n$-th step for the $j$-th sample is obtained by solving the regularized linear least-squares problem
\begin{equation}\label{problem}
    \min_{h}
\left\|
\mathcal F_k'[r_{j,n}]h
+
\mathcal F_k(r_{j,n})
-
u_j^{\infty,\delta}(\cdot,\cdot;k)
\right\|_Y^2
+
\alpha_{j,n}
\left\|
h+r_{j,n}-r_{\rm init}
\right\|_X^2,
\end{equation}
here $Y$ denotes the far‑field data space, $X$ the shape  space, $\alpha_{j,n}$ the regularization parameter and $r_{init}$ the initial shape. For computational tractability, we avoid solving for 
$h_{j,n}$ directly in the infinite-dimensional function space and instead expand it in a finite-dimensional basis. Specifically, we set
$h_{j,n}(\hat x)
=
\sum_{p=1}^{N_L} c_{j,n,p} \psi_p(\hat x),$
where $\{\psi_p\}_{p=1}^{N_L}$ denotes a basis of truncated spherical harmonics and $\{c_{j,n,p}\}_{p=1}^{N_L}$ are the coefficients to be determined. This parameterization effectively transforms the infinite-dimensional unknown $h_{j,n}$ at the $n$-th iteration into a finite-dimensional coefficient vector $\mathbf c_{j,n}
=
\bigl(c_{j,n,1},\ldots,c_{j,n,N_L}\bigr)^T$. For each basis function $\psi_p$, we consider the perturbation direction $h=\psi_p$ which induces the boundary displacement $V_{j,n,p}(\hat{x})=\psi_p(\hat{x})\hat{x}$. The corresponding shape derivative field $v_{j,n,p}$ then satisfies
\begin{equation}
    v_{j,n,p}
    =
    -(V_{j,n,p}\cdot\nu_{j,n})
    \frac{\partial u_{j,n}}{\partial \nu_{j,n}}
    =
    -\psi_p(\hat x)(\hat x\cdot\nu_{j,n})
    \frac{\partial u_{j,n}}{\partial \nu_{j,n}}
    \quad \text{on } \Gamma_{j,n}.
\end{equation}
By assembling the far-field derivatives corresponding to each basis function as column vectors, we obtain the Jacobian matrix
$J_{j,n}
=
\left[
\mathcal F_k'[r_{j,n}]\psi_1,\,
\mathcal F_k'[r_{j,n}]\psi_2,\,
\ldots,\,
\mathcal F_k'[r_{j,n}]\psi_{N_L}
    \right].$
Thus, the linearized system \eqref{linear} takes the form
\begin{equation}\label{38}
    \mathcal F_k(r_{j,n}+h_{j,n})
    \approx
    \mathcal F_k(r_{j,n})+J_{j,n}\mathbf c_{j,n}.
\end{equation}
Building on the linearization \eqref{38}, we employ the Gauss--Newton
method to solve \eqref{problem} to update the current parameters. Denote the current far-field residual by
$\mathbf b_{j,n}:=u_j^{\infty,\delta}(\cdot,\cdot;k)-\mathcal F_k(r_{j,n})$.
Therefore, the coefficient update at the $n$-th iteration is defined by the following regularized least-squares problem:
\begin{equation}\label{zuixiao}
\mathbf c_{j,n}
=
\operatorname*{arg\,min}_{\mathbf c}
\left\|
J_{j,n}\mathbf c-\mathbf b_{j,n}
\right\|_2^2
+
\alpha_{j,n}
\left\|
\mathcal R
\bigl(\mathbf c+\mathbf a_{j,n}-\mathbf a_{\rm init}\bigr)
\right\|_2^2,
\end{equation}
here $\mathbf a_{j,n}$ and $\mathbf a_{\rm init}$ are the coefficient vectors for the current radius $r_{j,n}$ and the initial shape, respectively. $\alpha_{j,n}>0$ is the regularization parameter and $\mathcal R$ is the coefficient factor associated with the chosen regularization norm $\|\cdot\|_X$. To accommodate complex-valued far-field data while retaining real-valued shape coefficients, we separate the real and imaginary parts and define
\begin{equation}\label{complex}
    \widetilde J_{j,n}
=
\begin{pmatrix}
\operatorname{Re} J_{j,n}\\
\operatorname{Im} J_{j,n}
\end{pmatrix},
\qquad
\widetilde{\mathbf b}_{j,n}
=
\begin{pmatrix}
\operatorname{Re}\mathbf b_{j,n}\\
\operatorname{Im}\mathbf b_{j,n}
\end{pmatrix}.
\end{equation}
Substituting \eqref{complex} into \eqref{zuixiao}, the coefficient update is then determined by the real-valued regularized least-squares problem
\begin{equation}\label{que}
   \mathbf c_{j,n}
=
\operatorname*{arg\,min}_{\mathbf c}
\left\|
\widetilde J_{j,n}\mathbf c-\widetilde{\mathbf b}_{j,n}
\right\|_2^2
+
\alpha_{j,n}
\left\|
\mathcal R
\bigl(\mathbf c+\mathbf a_{j,n}-\mathbf a_{\rm init}\bigr)
\right\|_2^2.
\end{equation}
The first-order optimality condition of \eqref{que} yields the  normal equation
\begin{equation}
    \left(
\widetilde J_{j,n}^T\widetilde J_{j,n}
+
\alpha_{j,n} \mathcal{R}^T\mathcal{R}
\right)\mathbf c_{j,n}
=
\widetilde J_{j,n}^T\widetilde{\mathbf b}_{j,n}
-
\alpha_{j,n} \mathcal{R}^T\mathcal{R}(\mathbf a_{j,n}-\mathbf a_{\rm init}).
\end{equation}
Once the solution $\mathbf c_{j,n}=(c_{j,n,1},\ldots,c_{j,n,N_L})^T$ is computed, the radius is updated as
$r_{j,n+1}=r_{j,n}+h_{j,n}$ with
$h_{j,n}(\hat x)=\sum_{p=1}^{N_L}c_{j,n,p}\psi_p(\hat x)$.

To enhance the inversion stability and progressively recover finer geometric structures, 
we adopt a frequency continuation strategy: low-frequency components primarily capture 
the global obstacle outline, while high-frequency components encode local details. 
Let the available frequencies be $k_1<k_2<\cdots<k_Q$. For the $j$-th sample, the 
inversion starts at the lowest frequency $k_1$ from the initial shape
\[
    r_{j,1,0}=r_{\rm init}.
\]
At the $q$-th frequency stage, regularized Gauss--Newton iterations are performed with 
$k=k_q$ to obtain the stage reconstruction $r_{j,q,N_q}$, where $N_q$ denotes the 
number of iterations at this frequency. For each subsequent frequency, the previous 
stage reconstruction is passed as the initial guess:
\begin{equation}
    r_{j,q+1,0}=r_{j,q,N_q},\qquad q=1,\ldots,Q-1.
\end{equation}
The process continues until the largest frequency $k_Q$ is reached. The final output
\[
    r_{j,{\rm rec}}:=r_{j,Q,N_Q}
\]
serves as the recovered radius corresponding to the $j$-th random sample $r_j$.

The above recovery procedure is performed independently for each sample. As these 
samples constitute independent realizations, the mean radius function can be estimated by 
\begin{equation}\label{mean}
    \widebar r_{\rm rec}(\hat x)
=
\frac{1}{M}
\sum_{j=1}^{M}
r_{j,{\rm rec}}(\hat x).
\end{equation}
The resulting recovered mean shape is then expressed as
\begin{equation}
    \widebar\Gamma_{\rm rec}
=
\{\widebar r_{\rm rec}(\hat x)\hat x:\hat x\in\mathbb S^2\}.
\end{equation}
This completes the reference shape recovery process. The above procedure is summarized in Algorithm~\ref{alg:MC-MFRLA-mean} which constitutes the first part of our proposed reconstruction method.
 
\begin{algorithm}[htbp]
\small
\caption{Monte Carlo-based multi-frequency recursive linearization algorithm for mean-shape recovery}
\label{alg:MC-MFRLA-mean}

\SetKwInOut{Input}{Input}
\SetKwInOut{Output}{Output}

\Input{
Noisy far-field data
$\{u_j^{\infty,\delta}(\hat{x}_m,d_t;k_q)\}$,
$j=1,\ldots,M$, $m=1,\ldots,N_{\rm obs}$,
$t=1,\ldots,N_{\rm inc}$, $q=1,\ldots,Q$;
frequencies $k_1<\cdots<k_Q$;
spherical harmonic truncation degree $L_{\max}$ and
$N_L=(L_{\max}+1)^2$;
initial guess $\boldsymbol{a}_{\rm init}$;
regularization matrix $\mathcal{R}$;
regularization parameters $\alpha_{j,q,n}$;
tolerance $\varepsilon$; maximum iteration number $T$.
}

\Output{
Recovered sample radii $\{r_{j,{\rm rec}}\}_{j=1}^M$,
recovered mean radius $\widebar r_{\rm rec}$
and recovered mean shape $\widebar\Gamma_{\rm rec}$.
}

Construct the truncated spherical harmonic basis
\[
\{\psi_p\}_{p=1}^{N_L}
=
\{Y_{\ell s}:0\leq\ell\leq L_{\max},\ -\ell\leq s\leq\ell\}.
\]

\For{$j=1,2,\ldots,M$}{

    Set the initial coefficient vector at the first frequency:
    \[
        \boldsymbol{a}_{j,1,0}
        =
        \boldsymbol{a}_{\rm init}.
    \]

    \For{$q=1,2,\ldots,Q$}{

        \If{$q>1$}{
            Set the initial guess at the current frequency by the previous reconstruction:
            \[
                \boldsymbol{a}_{j,q,0}
                =
                \boldsymbol{a}_{j,q-1}^{\rm out}.
            \]
        }

        Set $n=0$ and $\tau_0=1$\;

        \While{$\tau_n>\varepsilon$ and $n<T$}{

            Compute the current radius
            $r_{j,q,n}=r(\,\cdot\,;\boldsymbol{a}_{j,q,n})$\;

            Compute the residual and the shape derivative matrix:
            \[
            \boldsymbol{b}_{j,q,n}
            =
            u_j^{\infty,\delta}(\cdot,\cdot;k_q)
            -
            \mathcal{F}_{k_q}
            \bigl(r_{j,q,n}\bigr),
            \]
            \[
            J_{j,q,n}
            =
            \left[
            \mathcal{F}_{k_q}'
            \bigl[r_{j,q,n}\bigr]\psi_1,\,
            \mathcal{F}_{k_q}'
            \bigl[r_{j,q,n}\bigr]\psi_2,\,
            \ldots,\,
            \mathcal{F}_{k_q}'
            \bigl[r_{j,q,n}\bigr]\psi_{N_L}
            \right].
            \]

            Solve the regularized linearized problem:
            \[
            \boldsymbol{c}_{j,q,n}
            =
            \operatorname*{arg\,min}_{\boldsymbol{c}\in\mathbb{R}^{N_L}}
            \left\{
            \|J_{j,q,n}\boldsymbol{c}
            -\boldsymbol{b}_{j,q,n}\|_2^2
            +
            \alpha_{j,q,n}
            \left\|
            \mathcal{R}
            \bigl(
            \boldsymbol{a}_{j,q,n}
            +\boldsymbol{c}
            -\boldsymbol{a}_{\rm init}
            \bigr)
            \right\|_2^2
            \right\}.
            \]

            Update the coefficient vector:
            \[
            \boldsymbol{a}_{j,q,n+1}
            =
            \boldsymbol{a}_{j,q,n}
            +
            \boldsymbol{c}_{j,q,n},
            \qquad
            \tau_{n+1}
            =
            \frac{\|\boldsymbol{c}_{j,q,n}\|_2}
            {\max\{1,\|\boldsymbol{a}_{j,q,n}\|_2\}}.
            \]

            Set $n=n+1$\;
        }

        Set the output of the current frequency stage:
        \[
            \boldsymbol{a}_{j,q}^{\rm out}
            =
            \boldsymbol{a}_{j,q,n}.
        \]
    }

    Compute the recovered radius for the $j$-th sample:
    \[
        r_{j,{\rm rec}}
        =
        r(\,\cdot\,;\boldsymbol{a}_{j,Q}^{\rm out}).
    \]
}

Compute the empirical mean radius:
\[
\widebar r_{\rm rec}(\hat{x})
=
\frac{1}{M}
\sum_{j=1}^{M}
r_{j,{\rm rec}}(\hat{x}).
\]

Construct the recovered mean shape:
\[
\widebar\Gamma_{\rm rec}
=
\{\widebar r_{\rm rec}(\hat{x})\hat{x}:\hat{x}\in\mathbb S^2\}.
\]

\end{algorithm}

\subsection{The reconstruction of random statistics}\label{sub4}
\label{random statistics}

\subsubsection{Gaussian random field}
Building upon the parametric representation in \autoref{randomob}, when the shape perturbations are modeled as a Gaussian random field, we estimate the statistical parameters of the underlying field by applying a Monte Carlo (MC) strategy to the reconstructed samples. Specifically, we aim to estimate the eigenvalues $\lambda_\ell$ in \eqref{KL}, together with the variance parameter $\sigma^2$ and the correlation length $\ell_c$ appearing in \eqref{cov}.

For the isotropic Gaussian random field on $\mathbb{S}^2$,   with the recovered
mean radius $\widebar\Gamma_{\rm rec}$ obtained from Monte Carlo averaging, we first define the centered fluctuations:
\begin{equation}\label{flu}
    \rho_{j,{\rm rec}}(\hat x)
=
r_{j,{\rm rec}}(\hat x)-\widebar r_{\rm rec}(\hat x).
\end{equation}
Then by discretizing each recovered perturbation on the spherical grid $\{\hat x_i\}_{i=1}^{N}$ we obtain the sample vector $\boldsymbol{\rho}_{j,{\rm rec}}
=
\bigl(
\rho_{j,{\rm rec}}(\hat{x}_1),
\rho_{j,{\rm rec}}(\hat{x}_2),
\ldots,
\rho_{j,{\rm rec}}(\hat{x}_N)
\bigr)^T.$ Stacking all $M$ such vectors as columns gives the data matrix $R
=
\left[
\boldsymbol{\rho}_{1,{\rm rec}}\ 
\boldsymbol{\rho}_{2,{\rm rec}}\ 
\cdots\
\boldsymbol{\rho}_{M,{\rm rec}}
\right]
\in\mathbb R^{N\times M}.$ To preserve the $L^2(\mathbb S^2)$ inner product structure in the discrete setting, we introduce spherical integration weights. Define $W=\mathrm{diag}(w_1,\ldots,w_N)$ as the diagonal matrix of quadrature weights with  $w_i$ approximating the area element at the $i$-th node. We then form the weighted data matrix $\widetilde R
=
W^{1/2}R.$ The weighted covariance matrix is then defined as
\begin{equation}\label{Cemp}
    \widetilde C_{\rm emp}
=
\frac{1}{M-1}
W^{1/2}RR^TW^{1/2} =
\frac{1}{M-1}
\widetilde R\widetilde R^T.
\end{equation}
To avoid explicitly forming the potentially large  $N\times N$ matrix  $\widetilde C_{\rm emp}$, we compute the SVD of $\widetilde R$ directly:
\begin{equation}\label{SVD}
    \widetilde R
=
U\Sigma V^T.
\end{equation}
Substituting \eqref{SVD} into \eqref{Cemp}, we obtain
\begin{equation}
    \widetilde C_{\rm emp}
=
\frac{1}{M-1}
\widetilde R\widetilde R^T
=
U
\frac{\Sigma^2}{M-1}
U^T .
\end{equation}
Let $s_a$ denote the $a$-th singular value of $\widetilde R$. The individual empirical
eigenvalues of the weighted covariance matrix are then given by
\begin{equation}
    \widehat\mu_a^{\rm rec}
=
\frac{s_a^2}{M-1},
\qquad a=1,2,\ldots,\operatorname{rank}(\widetilde R).
\end{equation}
Since an isotropic Gaussian field on $\mathbb S^2$ has $2\ell+1$ identical eigenvalues
for each spherical harmonic degree $\ell$, we group the individual empirical eigenvalues
accordingly and define
\begin{equation}
    \widehat\lambda_\ell^{\rm rec}
=
\frac{1}{2\ell+1}
\sum_{a=\ell^2+1}^{(\ell+1)^2}
\widehat\mu_a^{\rm rec},
\qquad
\ell=0,1,\ldots,L_{\max},
\end{equation}
provided that $(\ell+1)^2\le \operatorname{rank}(\widetilde R)$. This gives the degree-wise KL spectrum used in the estimation of the Gaussian covariance parameters.

Given the reconstructed degree-wise KL eigenvalues 
$\{\widehat\lambda_\ell^{\rm rec}\}$, the next step is to infer the 
hyper-parameters $\sigma$ and $\ell_c$ of the covariance kernel. By construction, 
the covariance kernel of the random fluctuations is isotropic and therefore depends 
only on the angular separation $\gamma$ on $\mathbb S^2$, namely
$C_\rho(\hat x,\hat y)=C_\rho(\gamma)$ with 
$\cos\gamma=\hat x\cdot\hat y$. For this covariance operator, the degree-wise 
eigenvalues are given by
\begin{equation}\label{app}
    \lambda_\ell
=
2\pi
\int_0^\pi
C_\rho(\gamma)P_\ell(\cos\gamma)\sin\gamma\,d\gamma ,
\end{equation}
where $P_\ell$ denotes the Legendre polynomial of degree $\ell$. For short correlation lengths, the covariance kernel in \eqref{cov} can be locally 
approximated by 
$C_\rho(\gamma)\approx \sigma^2\exp(-\gamma^2/\ell_c^2)$.   Substituting this local form into \eqref{app} gives
\begin{equation}\label{apppro}
     \lambda_\ell
\approx
2\pi\sigma^2
\int_0^\infty
e^{-\gamma^2/\ell_c^2}
J_0\big((\ell+\tfrac12)\gamma\big)\gamma\,d\gamma
\approx
\pi\sigma^2\ell_c^2
\exp\left(
-\frac{\ell(\ell+1)\ell_c^2}{4}
\right),
\end{equation}
where $J_0$ is the Bessel function of the first kind of order zero. Taking the logarithm of both sides of \eqref{apppro}, we obtain
\begin{equation}\label{log}
    \log \lambda_\ell
\approx
\log(\pi\sigma^2\ell_c^2)
-
\frac{\ell_c^2}{4}\ell(\ell+1).
\end{equation}
Thus, by defining 
$x_\ell=\ell(\ell+1)$ and 
$y_\ell=\log \widehat\lambda_\ell^{\rm rec}$, where 
$\widehat\lambda_\ell^{\rm rec}$ denotes the reconstructed KL eigenvalue of degree 
$\ell$, the spectral approximation \eqref{log} reduces to the linear model
\begin{equation}
    y_\ell\approx A-Bx_\ell,
\end{equation}
where
$A=\log(\pi\sigma^2\ell_c^2)$ and 
$B=\ell_c^2/4$. By selecting a set of stable low-degree modes $\ell\in\mathcal I \subseteq\{0,1,\ldots,L_{\max}\}$, we perform the 
least-squares fitting
\begin{equation}
(A_{\rm fit},B_{\rm fit})
=
\operatorname*{arg\,min}_{A,B}
\sum_{\ell\in\mathcal I}
\left[
\log \widehat\lambda_\ell^{\rm rec}
-
A
+
B\ell(\ell+1)
\right]^2 .
\end{equation}
The correlation length is estimated from $B_{\rm fit}$, and the fluctuation amplitude 
is then recovered from $A_{\rm fit}$:
\begin{equation}
    \widehat\ell_c
=
2\sqrt{B_{\rm fit}},
\qquad
\widehat\sigma
=
\left(
\frac{e^{A_{\rm fit}}}
{\pi\widehat\ell_c^2}
\right)^{1/2}.
\end{equation}

\subsubsection{Non-Gaussian random field}When the fluctuations are non-Gaussian, the inversion procedure still
begins by independently reconstructing a deterministic obstacle for each
random sample. The recovered mean in \eqref{mean} is then approximated
by Monte Carlo averaging and the centered perturbations are computed
according to \eqref{flu}.

However, although the KL expansion can still be
constructed for the non-Gaussian perturbation field, its coefficients
are generally pairwise uncorrelated but not necessarily independent, therefore
the resulting spectral representation  captures only the
second-order structure and may not reflect higher-order statistical
dependence among the coefficients. For this reason, we do not regard
the KL spectrum as the primary quantity to be estimated. Instead, the
covariance structure is recovered directly from the centered
reconstructed perturbations. Specifically, for any
$\hat{x},\hat{y}\in\mathbb{S}^2$, the covariance kernel is approximated
by
\begin{equation}
    \widehat C_{\rho}^{\rm rec}(\hat x,\hat y)
=
\frac{1}{M-1}
\sum_{j=1}^{M}
\rho_{j,{\rm rec}}(\hat x)\rho_{j,{\rm rec}}(\hat y).
\end{equation}
Under the isotropy assumption, the covariance depends only on the
angular separation
\begin{equation}
\gamma
=
\arccos(\hat{x}\cdot\hat{y}).
\end{equation}
To obtain the estimated covariance function, the interval of possible
angular separations is divided into small subintervals. The covariance
values associated with direction pairs whose angular separations lie
in the same subinterval are then averaged, yielding an approximation
$\widehat C_\rho^{\rm rec}(\gamma)$ of the covariance as a function of angular
separation.

While the recovered covariance characterizes the second-order spatial
dependence of the non-Gaussian perturbations, it does not determine their
 probability law. To obtain additional distributional
information, we employ kernel density estimation (KDE)
\cite{wkeglarczyk2018kernel} to approximate the pointwise marginal
probability density functions. Specifically, for a fixed direction
$\hat{x}\in\mathbb{S}^2$, the reconstructed samples provide the
pointwise perturbation values
\begin{equation}
    \rho_{1,{\rm rec}}(\hat x),\ 
    \rho_{2,{\rm rec}}(\hat x),\ 
    \ldots,\ 
    \rho_{M,{\rm rec}}(\hat x).
\end{equation}
The corresponding marginal probability density is estimated by
\begin{equation}
    \widehat f_{\hat x}^{\rm rec}(s)
=
\frac{1}{M h_{\rm kde}}
\sum_{j=1}^{M}
\mathcal K\left(
\frac{
s-\rho_{j,{\rm rec}}(\hat x)
}{h_{\rm kde}}
\right),
\end{equation}
where $\mathcal K$ is a kernel function, $h_{\rm kde}>0$ denotes the bandwidth and $s$
represents a possible perturbation value in the direction $\hat{x}$.
This nonparametric procedure requires no \textit{a priori} assumption
that the perturbations follow a specific parametric distribution,
making it suitable for characterizing the pointwise marginal
distributions of non-Gaussian shape fluctuations. The overall statistical recovery procedure for both Gaussian and non-Gaussian perturbations is summarized as Algorithm~\ref{alg:statistical-recovery}.

\begin{algorithm}[htbp]
\small
\caption{Statistical recovery from reconstructed samples}
\label{alg:statistical-recovery}

\SetKwInOut{Input}{Input}
\SetKwInOut{Output}{Output}
\SetKw{Return}{return}

\Input{
Recovered sample radii $\{r_{j,{\rm rec}}\}_{j=1}^{M}$;
spherical quadrature grid $\{(\hat{x}_i,w_i)\}_{i=1}^{N}$;
maximum degree $L_{\max}$;
fitting index set $\mathcal I$;
KDE kernel $\mathcal K$ and bandwidth $h_{\rm kde}$.
}

\Output{
Recovered mean radius $\widebar r_{\rm rec}$;
centered perturbations $\{\rho_{j,{\rm rec}}\}_{j=1}^{M}$;
empirical covariance $\widehat C_{\rho}^{\rm rec}$;
and the corresponding Gaussian or non-Gaussian statistical estimates.
}

Compute the empirical mean radius and centered perturbations:
\[
\widebar r_{\rm rec}(\hat{x})
=
\frac{1}{M}
\sum_{j=1}^{M}
r_{j,{\rm rec}}(\hat{x}),
\qquad
\rho_{j,{\rm rec}}(\hat{x})
=
r_{j,{\rm rec}}(\hat{x})-\widebar r_{\rm rec}(\hat{x}).
\]

Form the perturbation matrix and its weighted version:
\[
R=
[\boldsymbol{\rho}_{1,{\rm rec}}\ \cdots\ \boldsymbol{\rho}_{M,{\rm rec}}],
\qquad
\boldsymbol{\rho}_{j,{\rm rec}}
=
\bigl(
\rho_{j,{\rm rec}}(\hat{x}_1),\ldots,
\rho_{j,{\rm rec}}(\hat{x}_N)
\bigr)^T,
\]
\[
W=\operatorname{diag}(w_1,\ldots,w_N),
\qquad
\widetilde R=W^{1/2}R.
\]

Compute the empirical covariance:
\[
\widehat C_{\rho}^{\rm rec}(\hat{x},\hat{y})
=
\frac{1}{M-1}
\sum_{j=1}^{M}
\rho_{j,{\rm rec}}(\hat{x})
\rho_{j,{\rm rec}}(\hat{y}).
\]

\If{a Gaussian perturbation model is assumed}{

    Compute the SVD of the weighted perturbation matrix:
    \[
       \widetilde R=U\Sigma V^T,
\qquad
\widehat\mu_a^{\rm rec}
=
\frac{s_a^2}{M-1},
\qquad
a=1,\ldots,\operatorname{rank}(\widetilde R).
    \]

    Group the individual empirical eigenvalues according to the multiplicity 
    $2\ell+1$:
    \[
        \widehat\lambda_\ell^{\rm rec}
        =
        \frac{1}{2\ell+1}
        \sum_{a=\ell^2+1}^{(\ell+1)^2}
        \widehat\mu_a^{\rm rec},
        \qquad
        \ell=0,\ldots,L_{\max},
        \quad
        (\ell+1)^2\leq r.
    \]

    Estimate the Gaussian covariance parameters by fitting
    \[
        \log \widehat\lambda_\ell^{\rm rec}
        \approx
        A-B\ell(\ell+1),
        \qquad \ell\in\mathcal I.
    \]

    Set
    \[
        \widehat\ell_c=2\sqrt{B_{\rm fit}},
        \qquad
        \widehat\sigma
        =
        \left(
        \frac{e^{A_{\rm fit}}}
        {\pi\widehat\ell_c^2}
        \right)^{1/2}.
    \]
}

\ElseIf{a non-Gaussian perturbation model is considered}{

    Estimate the angular covariance curve by averaging 
    $\widehat C_{\rho}^{\rm rec}(\hat{x}_i,\hat{x}_{i'})$ over direction pairs with 
    similar angular separation:
    \[
        \gamma_{ii'}
        =
        \arccos(\hat{x}_i\cdot\hat{x}_{i'}).
    \]

    Estimate representative marginal densities by KDE:
    \[
        \widehat f_{\hat{x}}^{\rm rec}(s)
        =
        \frac{1}{M h_{\rm kde}}
        \sum_{j=1}^{M}
        \mathcal K
        \left(
        \frac{s-\rho_{j,{\rm rec}}(\hat{x})}
        {h_{\rm kde}}
        \right).
    \]
}

\Return{
$\widebar r_{\rm rec}$,
$\{\rho_{j,{\rm rec}}\}_{j=1}^{M}$,
$\widehat C_{\rho}^{\rm rec}$
and the corresponding Gaussian or non-Gaussian statistical estimates
}\;

\end{algorithm}

\subsection{Statistical uniqueness for random obstacles}\label{sub1}
This section establishes a statistical uniqueness result for the inverse random  obstacle problem. Under suitable assumptions, the probability law of the uncertain obstacle boundary is uniquely determined by the distributional information contained in the random far-field data. As a consequence, the mean shape and the associated  statistics are uniquely determined.

Let $X = C^m(\mathbb S^2;\mathbb{R})$ be the shape space of radial functions where $m$ ensures the required smoothness of the star-shaped obstacle boundary. We consider the
class of radial functions $\mathcal A\subseteq X$ where
\begin{equation}\label{admiss}
\mathcal{A}
=
\left\{
r\in X:
0<r_-\leq r(\hat{x})\leq r_+
\ \text{for all } \hat{x}\in\mathbb{S}^2,
\quad
\|r\|_{C^m(\mathbb{S}^2)}\leq M
\right\},
\end{equation}
here $r_->0$, $r_+<\infty$, and $M<\infty$ are fixed constants. Each
$r\in\mathcal{A}$ determines the bounded star-shaped obstacle 
\begin{equation}
D_r
=
\left\{
x_0+\tau\hat{x}:
\hat{x}\in\mathbb{S}^2,\ 
0\leq \tau<r(\hat{x})
\right\},
\end{equation}
whose boundary is given by
\begin{equation}
\Gamma_r
=
\partial D_r
=
\left\{
x_0+r(\hat{x})\hat{x}:
\hat{x}\in\mathbb{S}^2
\right\}.
\end{equation}
For a fixed wavenumber $k>0$ and incident directions
$d_1,\ldots,d_{N_{\rm inc}}\in\mathbb S^2$, define the far-field space
\begin{equation}\label{farfieldspace}
Y
=
\prod_{t=1}^{N_{\rm inc}}
L^2(\mathbb S^2).
\end{equation}
and the deterministic far-field map
\begin{equation}\label{forwardmap}
\mathcal F:\mathcal A\to Y,
\qquad
\mathcal F(r)
=
\bigl(
u_r^\infty(\cdot,d_1;k),
\ldots,
u_r^\infty(\cdot,d_{N_{\rm inc}};k)
\bigr).
\end{equation}
Let $(\Omega_i,\mathscr{F}_i,\mathbb{P}_i)$, $i=1,2$, be two probability
spaces, and let
\begin{equation}
R_i:
(\Omega_i,\mathscr{F}_i)
\longrightarrow
(\mathcal{A},\mathcal{B}(\mathcal{A})),
\qquad i=1,2,
\end{equation}
be measurable random radial functions. For each $\omega_i\in\Omega_i$,
the realization $R_i(\omega_i)$ determines a deterministic star-shaped
obstacle. The corresponding random far-field data are defined by
\begin{equation}
U_i
=
\mathcal{F}(R_i),
\qquad i=1,2.
\end{equation}
Thus, each $U_i$ is a $Y$-valued random variable whose realizations
consist of the far-field patterns associated with the prescribed
incident directions. We are now in a position to state the statistical
uniqueness result:
\begin{theorem}\label{theo}
Let $\mathcal{A}$, $Y$, $R_i$, and $U_i$, $i=1,2$, be defined as above.
Suppose that at a fixed wavenumber $k>0$, the complete-aperture far-field
data are available for the prescribed incident directions
$d_1,\ldots,d_{N_{\rm inc}}\in\mathbb S^2$. If the corresponding $Y$-valued
random far-field data have the same probability law,
\begin{equation}\label{cond1}
\mathcal{L}(U_1)=\mathcal{L}(U_2)
\qquad \text{on }Y,
\end{equation}
then
\begin{equation}
\mathcal{L}(R_1)=\mathcal{L}(R_2)
\qquad \text{on }\mathcal{A}.
\end{equation}
In particular,
\begin{equation}
\mathbb{E}[R_1]
=
\mathbb{E}[R_2]
\qquad \text{in }X
\end{equation}
\end{theorem}

\begin{proof}
Since $m<\infty$ and $\mathbb{S}^2$ is compact, the space $X=C^m(\mathbb{S}^2;\mathbb{R})$ endowed with the $C^m$-norm is a separable Banach space and hence a
Polish space. Moreover, $Y$ defined in \eqref{farfieldspace} is a finite product of separable Hilbert spaces and is therefore also
Polish. The constraints defining $\mathcal{A}$ are closed in $X$.
Consequently, $\mathcal{A}$ is a closed, and hence Polish, subspace of
$X$.

Next, we demonstrate that the forward map $\mathcal F:\mathcal A\to Y$ in \eqref{forwardmap} is continuous and injective. The continuity of $\mathcal F$ directly follows from the domain differentiability of the far-field operator under boundary
deformations \cite{MR1203018}.  Furthermore, under the constraints defining $\mathcal{A}$ in \eqref{admiss} and the prescribed
measurement configuration, the Colton–Sleeman uniqueness result
\cite{MR729385} applies. Hence, for any $r_1,r_2\in\mathcal{A}$,
\begin{equation}
\mathcal{F}(r_1)=\mathcal{F}(r_2)
\quad\Longrightarrow\quad
r_1=r_2.
\end{equation}
Therefore, $\mathcal{F}$ is continuous and injective on $\mathcal{A}$.

We next establish the Borel measurability of
$\mathcal{F}^{-1}:\mathcal{F}(\mathcal{A})\to\mathcal{A}$.
By definition, $\mathcal F^{-1}$ is Borel measurable if for every $B\in\mathcal B(\mathcal A)$, one has $(\mathcal F^{-1})^{-1}(B)\in\mathcal B(\mathcal F(\mathcal A))$. For such a Borel set $B$, we have
\begin{equation}
(\mathcal F^{-1})^{-1}(B)
=
\left\{
y\in \mathcal F(\mathcal A):\mathcal F^{-1}(y)\in B
\right\}
=
\mathcal F(B).
\end{equation}
Thus, it remains to show that $\mathcal F(B)$ is Borel in $\mathcal F(\mathcal A)$ for every $B\in\mathcal B(\mathcal A)$. This follows directly from the Lusin--Souslin theorem \cite{Kechris1995} since $\mathcal A$ is a Borel subset of the Polish space $X$, $Y$ is Polish, and $\mathcal F:\mathcal A\to Y$ is continuous and injective. Therefore,
\begin{equation}
\mathcal F^{-1}:\mathcal F(\mathcal A)\to\mathcal A
\end{equation}
is Borel measurable. 

Since $U_i=\mathcal{F}(R_i)$, each $U_i$ takes values in
$\mathcal{F}(\mathcal{A})$ and
\begin{equation}
R_i
=
\mathcal{F}^{-1}(U_i)
\qquad
\mathbb{P}_i\text{-a.s.},
\qquad i=1,2.
\end{equation}
The condition \eqref{cond1} together with the Borel measurability of $\mathcal{F}^{-1}$ yields
\begin{equation}\label{75}
\begin{aligned}
\mathcal{L}(R_1)=
\mathcal{L}\bigl(\mathcal{F}^{-1}(U_1)\bigr) =
\mathcal{L}\bigl(\mathcal{F}^{-1}(U_2)\bigr) =
\mathcal{L}(R_2).
\end{aligned}
\end{equation}

It remains to show the equality of the mean radial functions. Since $R_i$ takes values in $\mathcal{A}$, the uniform bound in the
definition of $\mathcal{A}$ gives
\begin{equation}
\int_{\Omega_i}
\|R_i(\omega_i)\|_X^p
\,d\mathbb{P}_i(\omega_i)
\leq M^p,
\qquad p=1,2,\quad i=1,2.
\end{equation}
Hence, $R_i\in L^2(\Omega_i;X)\subset L^1(\Omega_i;X)$
and the Bochner expectation of $R_i$ is well defined. Let $\mu_i=\mathcal{L}(R_i),$ $i=1,2.$
By the integration formula for pushforward measures and the equality $\mu_1=\mu_2$ in \eqref{75}, we obtain
\begin{equation}
\begin{aligned}
\mathbb E[R_1]
&=
\int_{\Omega_1}R_1(\omega_1)\,d\mathbb P_1(\omega_1)
=
\int_{\mathcal A}r\,d\mu_1(r) 
=
\int_{\mathcal A}r\,d\mu_2(r)
=
\int_{\Omega_2}R_2(\omega_2)\,d\mathbb P_2(\omega_2)
=
\mathbb E[R_2]
\qquad \text{in }X.
\end{aligned}
\end{equation}
This completes the proof.
\end{proof}

 The equality in law established in \autoref{theo} further implies the uniqueness of the second-order statistics of the random radial function, as stated in the following corollary:
\begin{corollary}
\label{cor:second-order-uniqueness}
    Under the assumptions of \autoref{theo}, let
\begin{equation}
\bar{r}_i=\mathbb{E}[R_i],
\qquad
\rho_i=R_i-\bar{r}_i,
\qquad i=1,2.
\end{equation}
Then the covariance kernels of the centered random radial functions
coincide:
\begin{equation}
C_{\rho_1}(\hat{x},\hat{y})
=
C_{\rho_2}(\hat{x},\hat{y}),
\qquad
\hat{x},\hat{y}\in\mathbb{S}^2.
\end{equation}
Consequently, their pointwise variance functions and the corresponding
covariance operators on $L^2(\mathbb{S}^2)$ are identical. Hence, the
KL eigenvalues coincide, including their multiplicities, and the
associated eigenspaces are the same.
\end{corollary}
\begin{proof}

Since $\mu=\mathcal{L}(R_1)=\mathcal{L}(R_2)$ and 
$\bar{r}=\mathbb{E}[R_1]=\mathbb{E}[R_2]$,  for $i=1,2$ and
$\hat{x},\hat{y}\in\mathbb{S}^2$, 
\begin{equation}
C_{\rho_i}(\hat{x},\hat{y})
=
\int_{\mathcal{A}}
\bigl(r(\hat{x})-\bar{r}(\hat{x})\bigr)
\bigl(r(\hat{y})-\bar{r}(\hat{y})\bigr)
\,d\mu(r).
\end{equation}
Hence the covariance kernels coincide. In particular, by setting
$\hat{x}=\hat{y}$, the corresponding pointwise variance functions are
also identical. Therefore, the associated covariance operators
are uniquely determined. Consequently,  their eigenvalues including multiplicities coincide.
\end{proof}

\autoref{theo} shows that the probability law of the random far-field
data uniquely determines that of the random radial function. Under the
zero-mean perturbation model
\begin{equation}
R=\bar{r}+\rho,
\qquad
\mathbb{E}[\rho]=0,
\end{equation}
one has $\mathbb{E}[R]=\bar{r}$. Thus, the unique determination of the
Bochner mean of $R$ yields the uniqueness of the deterministic reference
shape. In addition, Corollary \autoref{cor:second-order-uniqueness} establishes
the unique determination of the covariance kernel, pointwise variance
function, covariance operator and KL eigenvalues of the perturbation
field. These results provide a theoretical foundation for the
statistical quantities recovered by the proposed reconstruction method.

\begin{remark}
It should be noted that \eqref{cond1} is a relatively strong
statistical assumption, as it requires equality of the probability
laws of the $Y$-valued random far-field data rather than only of
finitely many statistical quantities. In practice, these laws can only
be approximated from  a sufficiently
large number of independent realizations. This
formulation reflects the difficulty arising from the nonlinear forward
map $\mathcal{F}$ and the underlying deterministic inverse problem.
Establishing uniqueness under weaker statistical assumptions remains
an important topic for future study.
\end{remark}

\section{Numerical experiments} \label{Numerical_experiments}
In this section, we present four numerical studies to demonstrate the performance of the proposed reconstruction method. The first three examples focus on the quantitative inversion of Gaussian random fields, while the last one  considers obstacles with non‑Gaussian random perturbations:
\begin{enumerate}

\item The first example investigates the effect of sample size on
statistical recovery. Since the proposed method relies on Monte Carlo
estimation, sufficiently many reconstructed samples are required for
reliable recovery of the mean shape and covariance-related quantities,
whereas too few samples may reduce the accuracy of these estimates.

\item The second example considers several representative geometries
and demonstrates the effectiveness of the proposed algorithm in
recovering both deterministic shape features and the associated
random-field statistics.

\item The third example considers a more complex scatterer and examines
the effectiveness of the multi-frequency continuation strategy.
Comparisons with single-frequency reconstructions demonstrate that
multiple frequencies improve both shape recovery and statistical
estimation.

\item The final example addresses non-Gaussian random perturbations
generated by a digital-filter-based construction. It demonstrates that
the proposed method is not restricted to Gaussian random fields and can
also recover important statistical features of more general
perturbations, including empirical covariance structures and marginal
radius distributions at selected directions.
\end{enumerate}

\subsection{Experiment Setup} \label{experiment_setup}


In the following numerical experiments, unless otherwise specified, the synthetic far-field data are generated using six incident directions $d \in \{\pm e_1,\pm e_2,\pm e_3\}$ where \(e_1,e_2,e_3\) are the standard Cartesian basis vectors in \(\mathbb{R}^3\). For simpler geometries, fewer incident directions may also be sufficient. To avoid the inverse crime, the synthetic far-field data are generated by a combined-field potential ansatz on a finer spatial discretization rather than by the $\frac{\partial u}{\partial \nu}$-based Neumann-data formulation used in \autoref{sub2}. In addition, distinct spatial discretizations are employed and the shape representation as well as the number of unknown parameters are chosen differently for the forward and inverse problems. Precisely, the forward data are computed on a $30\times 60$ $(\mu,\phi)$-grid while the inverse reconstruction is carried out on a coarser $28\times 56$ grid. In the Gaussian random field case, the forward random perturbation is generated by a KL expansion truncated at $L_{\rm KL}=10$ (121 modes), whereas the inverse radial update is expanded only up to $L_{\rm inv}=7$ (64 unknown coefficients). The eigenvalues $\lambda_\ell$ in the forward scattering are computed by the Funk--Hecke formula \cite{dai2013approximation}. All synthetic far-field data are corrupted with 5\% relative noise. Following the decreasing regularization strategy of the iteratively regularized Gauss–Newton method, we set
\[
 \alpha_n=\gamma_n\frac{\operatorname{trace}(J_n^TJ_n)}{N_{\rm coeff}},\qquad\gamma_0=3\times10^{-2}\qquad\gamma_{n+1}\approx \frac13\gamma_n .\]
 The trace factor is used to normalize the scale of the linearized system so that the prescribed sequence $\gamma_n$ gives a comparable relative regularization strength across different frequency stages. The regularization matrix $\mathcal R$ is chosen to be the Sobolev-type diagonal matrix
$\mathcal R=\operatorname{diag}\left((1+\ell^2)^{s/2}\right)$ as in \cite{MR3933278}. We set  $ s=1.5$. The parameters $\eta_1$ and $\eta_2$ are both set to $k/2$ where $k$ denotes the frequency at the current iteration stage.
For the basic test geometries considered in the numerical experiments, we use several representative smooth star-shaped shapes, as illustrated in \autoref{simpleob}. Their radial parametrizations are given as follows. Let
\[
\hat x(\theta,\phi)
=
(\sin\theta\cos\phi,\sin\theta\sin\phi,\cos\theta),
\qquad
x(\theta,\phi)=r(\theta,\phi)\hat x(\theta,\phi).
\]
\begin{figure}
    \centering
       \includegraphics[width=\linewidth]{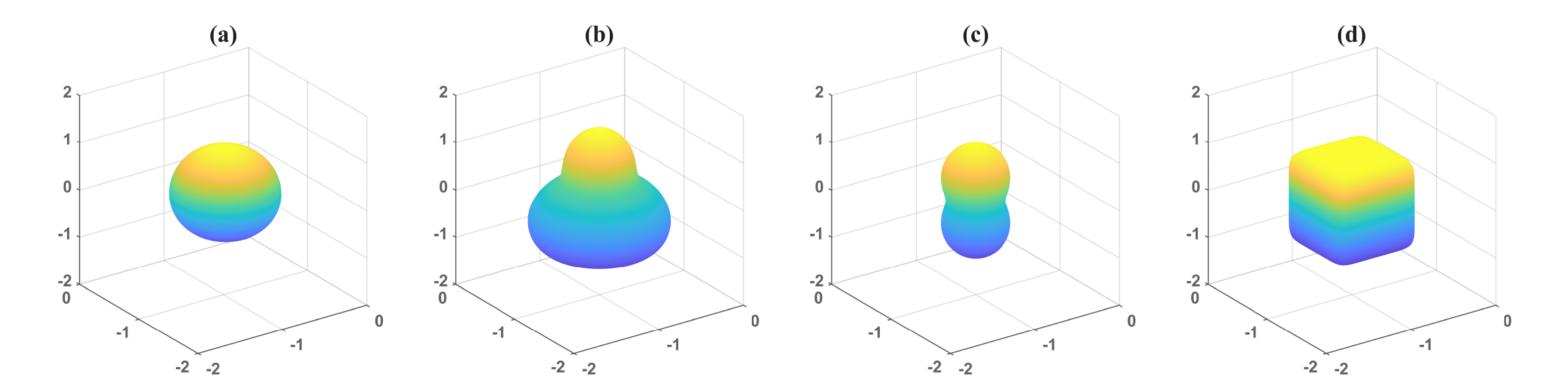}
        \caption{\scriptsize Several representative  obstacle shapes. (a): egg-shaped (b): pear-shaped. (c): peanut-shaped. (d)  rounded cube-shaped}
        \label{simpleob}
    \end{figure}
 \noindent Then the radius functions are specified by
{\small
\begin{enumerate}
    \item \textbf{Egg-shaped obstacle:}
    \(r_{\rm egg}(\theta,\phi)
    =
    1+0.18\cos\theta-0.08\cos(2\theta)
    +0.05\sin^2\theta\cos\phi.\)

    \item \textbf{Pear-shaped obstacle:}
    \(r_{\rm pear}(\theta)
    =
    0.60\sqrt{4.25+2\cos(3\theta)}.\)

    \item \textbf{Peanut-shaped obstacle:}
    \(r_{\rm peanut}(\theta)
    =
    1.17\sqrt{\cos^2\theta+0.25\sin^2\theta}.\)

    \item \textbf{Rounded cube-shaped obstacle:}
    \(r_{\rm cube}(\theta,\phi)
    =
    \dfrac{0.95}{
    \left(
    |\sin\theta\cos\phi|^8+
    |\sin\theta\sin\phi|^8+
    |\cos\theta|^8
    \right)^{1/8}}.\)
\end{enumerate}
}

\begin{figure}
    \centering
       \includegraphics[width=0.8\linewidth]{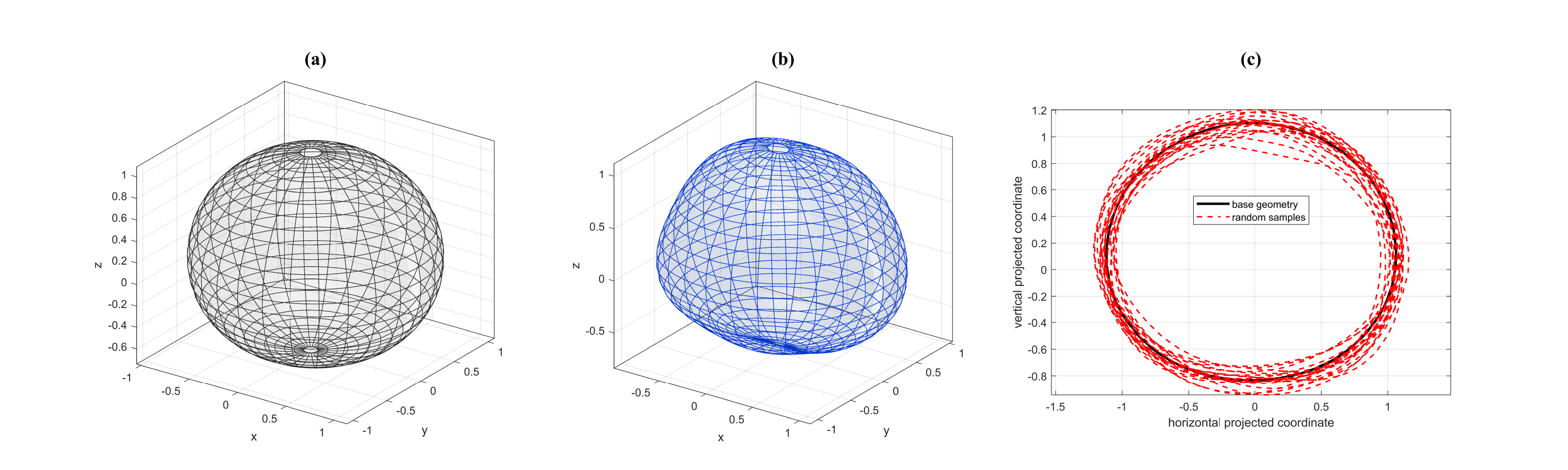}
        \caption{\scriptsize Random realizations of egg-like scatterer. (a): true geometry (without random perturbations). (b): a sample realization with random perturbations. (c): cross-sectional comparison between the perturbed sample (dashed line) and the true shape (solid line).}
        \label{eggg}
    \end{figure}
 \subsection{Example 1: role of  sample size in reconstruction}
 \label{example:circle}
This numerical example considers Gaussian random perturbations of an egg-shaped scatterer. Random realizations are generated with parameters $\sigma=0.06$ and $\ell_c=0.8$, corresponding to a relatively mild perturbation regime, referred to as case~\uppercase\expandafter{\romannumeral1} below. A rougher perturbation regime will be considered later and denoted by case~\uppercase\expandafter{\romannumeral2}. Several realizations of the resulting random scatterer are shown in \autoref{eggg}, and \autoref{eggg}(c) presents a cross-sectional comparison between the perturbed surfaces and the reference geometry. The figure indicates that, in case~\uppercase\expandafter{\romannumeral1}, the random perturbation has a relatively small amplitude compared with the overall size of the scatterer.

We first apply the proposed algorithm  to reconstruct the random egg-shaped scatterer. Due to the moderate geometric complexity of this example, far-field data at four wavenumbers ($k=3,4,5,6$) are sufficient to produce the accurate reconstructions shown in \autoref{reegg}.

\begin{figure}
    \centering
       \includegraphics[width=0.7\linewidth]{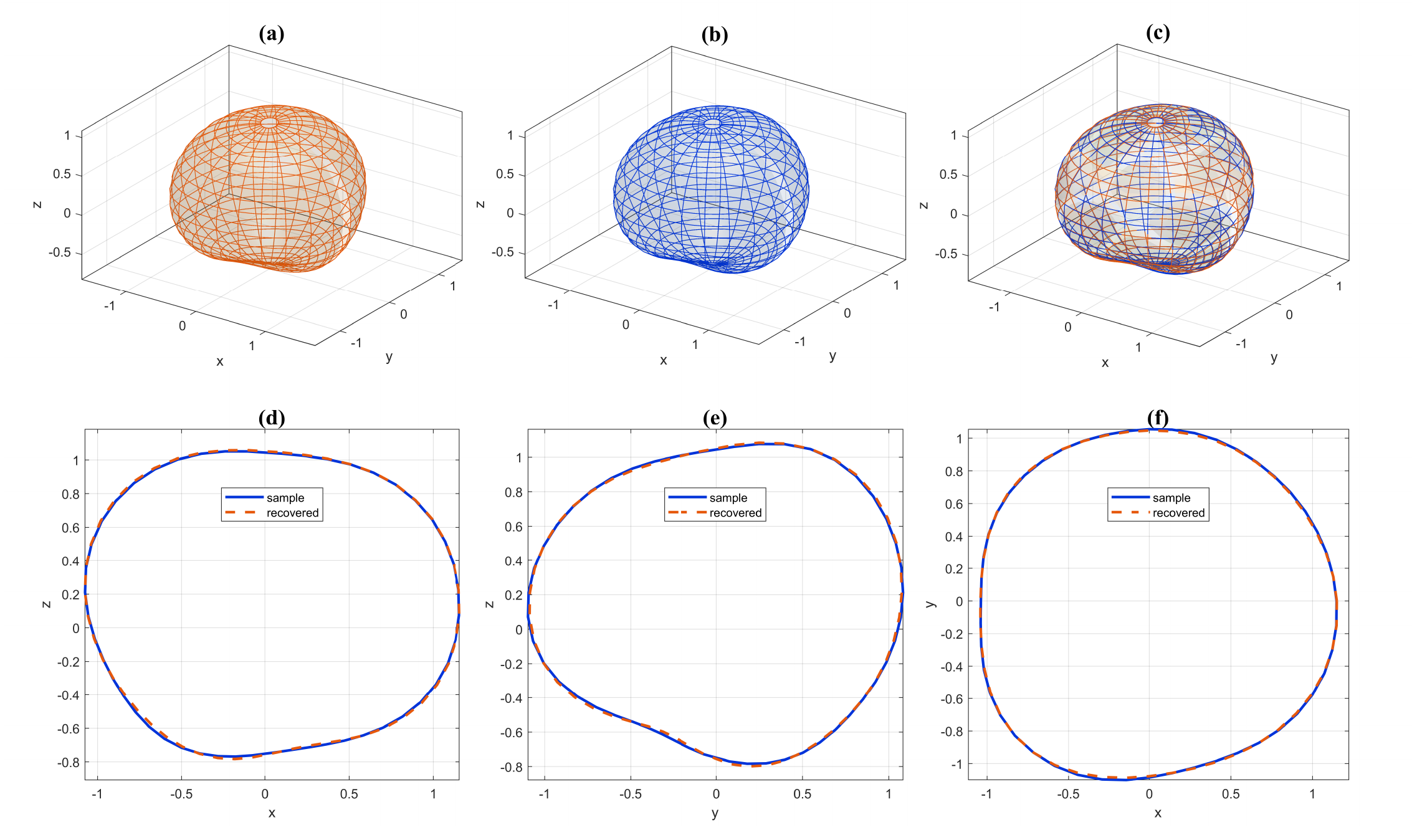}
        \caption{\scriptsize Reconstruction comparison. (a): reconstructed shape of one  sample. (b): random sample. (c): comparison between the sample and its reconstruction. (d): $x$-$z$ section comparison, $\phi$ = 0. (e): $y$-$z$ section comparison, $\phi$ = 1.57. (f): $x$-$y$ equatorial section comparison.}
        \label{reegg}
    \end{figure}

Since the proposed method relies on a Monte Carlo strategy, the number of samples plays a crucial role in stable statistical recovery. To illustrate this effect, \autoref{KLeggnew}(c) reports the relative error of the reconstructed mean shape as the number of samples increases, shown by the blue solid curve. The error decreases rapidly for small sample sizes and then stabilizes when sufficiently many samples are used, indicating that the sample mean provides an accurate approximation of the underlying reference geometry. The red dashed curve shows the corresponding discrepancy between the recovered and true KL eigenvalues. Its overall decreasing trend further demonstrates that an adequate sample size is essential for reliably recovering not only the mean shape but also the statistical features of the random perturbation field.

\begin{figure}
    \centering
       \includegraphics[width=0.9\linewidth]{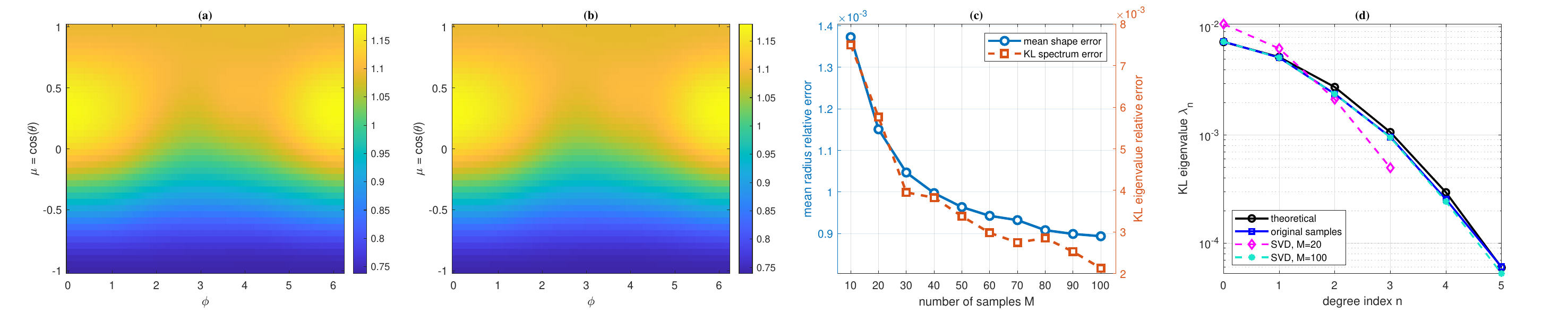}
        \caption{\scriptsize (a): true mean radius
  (b): recovered mean radius using all available reconstructed samples, $M$ = 100.
  (c): sample-size effect: mean-shape relative error and KL-eigenvalue relative error versus $M$.
  (d):  KL eigenvalues estimation with 20 samples and 100 samples. }
        \label{KLeggnew}
    \end{figure}

For the recovery of the covariance hyper-parameters $\sigma$ and $\ell_c$ defined in \eqref{cov}, the comparison results are reported in \autoref{tab:hyperparameter_estimat}. The true (preset) values denote the prescribed theoretical input hyper-parameters, the reference values are obtained by applying our inversion method to the true geometry (serving as a reference), and the estimate values are the sample-based reconstructed hyper-parameters.
 These results further confirm that a sufficiently large number of samples is crucial for achieving reliable shape reconstruction, KL-spectrum estimation and hyper-parameter recovery. 
 
 We next consider a rougher perturbation regime by increasing the variance parameter to $\sigma=0.08$ and reducing the correlation length to $\ell_c=0.6$ (case \uppercase\expandafter{\romannumeral2}). This setting leads to stronger local shape fluctuations and weaker spatial correlation between different  points, as can be clearly observed from the sample realization shown in  \autoref{exam11}(a).

\begin{table}[htbp]
\centering
\tiny
\renewcommand{\arraystretch}{1.18}
\setlength{\tabcolsep}{4pt}
\setlength{\heavyrulewidth}{1.05pt}
\setlength{\lightrulewidth}{0.55pt}
\setlength{\cmidrulewidth}{0.55pt}

\begin{tabularx}{\textwidth}{
ll
>{\centering\arraybackslash}p{1.7cm}
>{\centering\arraybackslash}p{1.5cm}
>{\centering\arraybackslash}X
>{\centering\arraybackslash}X
>{\centering\arraybackslash}X
}
\toprule
&& \makecell{True(Preset)}
& Reference
& \makecell{Estimated(20 samples)}
& \makecell{Estimated(50 samples)}
& \makecell{Estimated(100 samples)} \\
\midrule

\multirow{2}{*}{case \uppercase\expandafter{\romannumeral1}}
& $\sigma$  & 0.0600 & 0.0588 & 0.0570 (95\%) & 0.0580 (97\%)& \textbf{0.0579 (97\%)} \\
& $\ell_c$ & 0.80   & 0.78  & 1.01 (74\%)  & 0.85  (94\%) & \textbf{0.81 (99\%)} \\
\midrule

&& \makecell{True (Preset)}
& Reference
& \makecell{Estimated(50 samples)}
& \makecell{Estimated(100 samples)}
& \makecell{Estimated(300 samples)} \\
\midrule

\multirow{2}{*}{case \uppercase\expandafter{\romannumeral2}}
& $\sigma$  & 0.0800 & 0.0780 & 0.0780 (98\%) & 0.0765 (96\%)& \textbf{0.0773 (97\%)} \\
& $\ell_c$ & 0.60   & 0.60   & 0.72  (80\%) & 0.66 (90\%)  & \textbf{0.62 (97\%)} \\
\bottomrule
\end{tabularx}

\caption{\scriptsize  Comparison of estimated hyper-parameters under two settings with different sample sizes. True (Preset) denotes the ground-truth covariance parameters, while Reference denotes estimates recovered via inversion with the true geometric configuration.}
\label{tab:hyperparameter_estimat}
\end{table}

\begin{figure}
    \centering
       \includegraphics[width=0.9\linewidth]{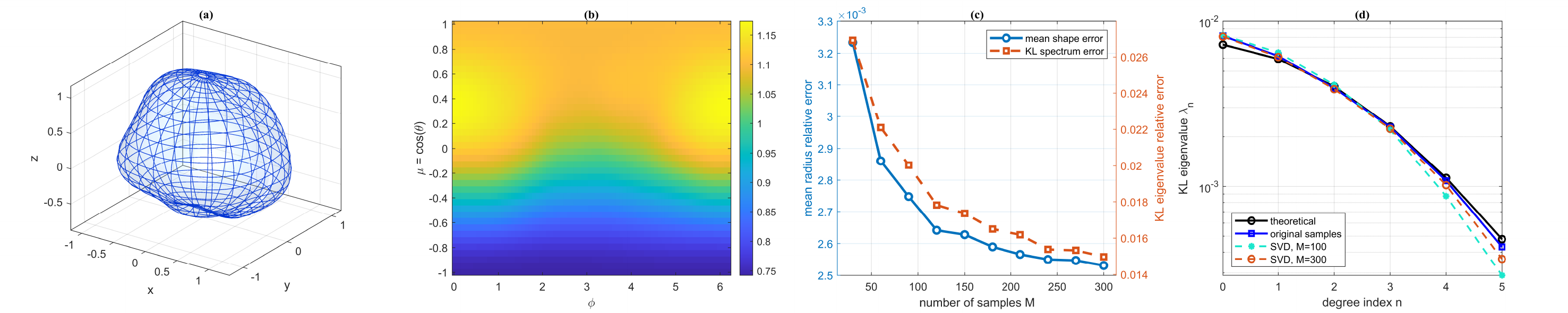}
        \caption{\scriptsize  (a): original random sample geometry($\# 1$).
  (b): recovered mean radius using 300 reconstructed samples.
  (c): sample-size effect: mean-shape error and KL eigenvalue error versus $M$. (d): KL eigenvalues estimation with 100 and 300 samples}
        \label{exam11}
    \end{figure}

\autoref{exam11}(b) shows the sample-averaged reconstructed radius under case \uppercase\expandafter{\romannumeral2}. The corresponding hyper-parameter recovery results are also reported in \autoref{tab:hyperparameter_estimat}. Compared with the previous regime, the larger fluctuation amplitude and shorter correlation length make the statistical recovery more challenging. Consequently, a larger number of samples is needed to maintain sufficient accuracy in the reconstructed mean shape, KL spectrum and covariance hyper-parameters.

\begin{figure}
\begin{center}
    \begin{overpic}[width=0.9\textwidth, tics=10]{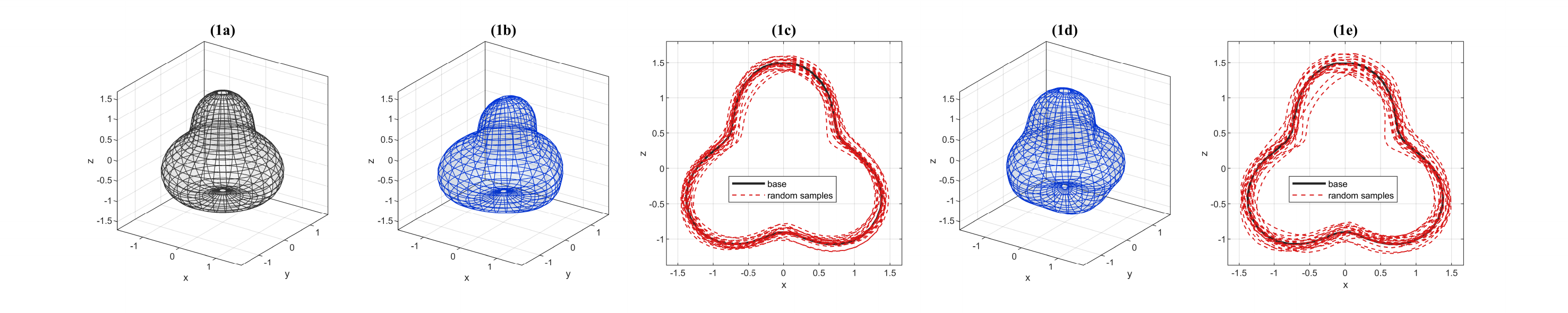}
    \end{overpic}
\end{center}
 \begin{center}
    \begin{overpic}[width=0.9\textwidth, tics=10]{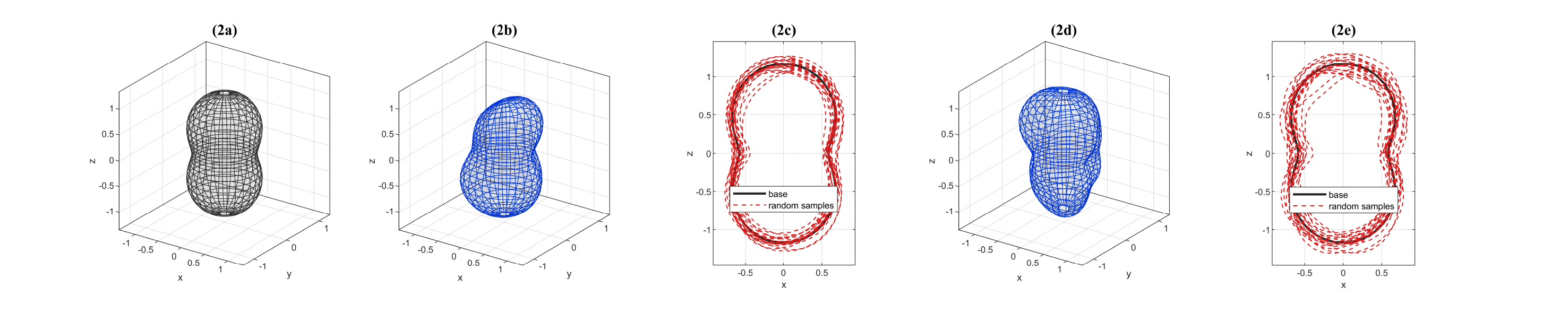}
    \end{overpic}
\end{center}
 \begin{center}
    \begin{overpic}[width=0.9\textwidth, tics=10]{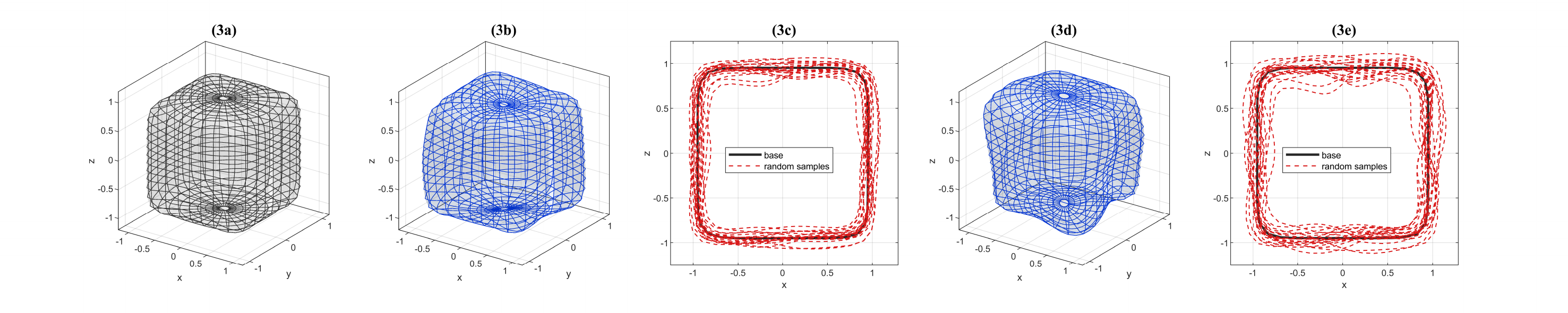}
    \end{overpic}
\end{center}
 \vspace{-0.5cm}
    \caption{\scriptsize (1a)(2a)(3a): true shape. (1b)(2b)(3b): one random sample realization under parameter setting \uppercase\expandafter{\romannumeral1}  ($\sigma$ = 0.06, $\ell_c$= 0.8). (1c)(2c)(3c): $x$-$z$ section of all random samples under parameter set \uppercase\expandafter{\romannumeral1}  ($\sigma$ = 0.06, $\ell_c$= 0.8). (1d)(2d)(3d): one random sample realization under case \uppercase\expandafter{\romannumeral2} ($\sigma$ = 0.08, $\ell_c$= 0.6). (1e)(2e)(3e): $x$-$z$ section of all random samples under case \uppercase\expandafter{\romannumeral2}  ($\sigma$ = 0.08, $\ell_c$= 0.6).}
    \label{exam2}
      \end{figure}

\subsection{Example 2: reconstruction for several representative shapes}

This example extends the quantitative inversion framework to the remaining scatterer configurations illustrated in \autoref{simpleob}, with the uncertain perturbations still modeled as a Gaussian isotropic random field. 


 \begin{figure}
    \begin{center}
    \begin{overpic}[width=0.9\textwidth,tics=10]{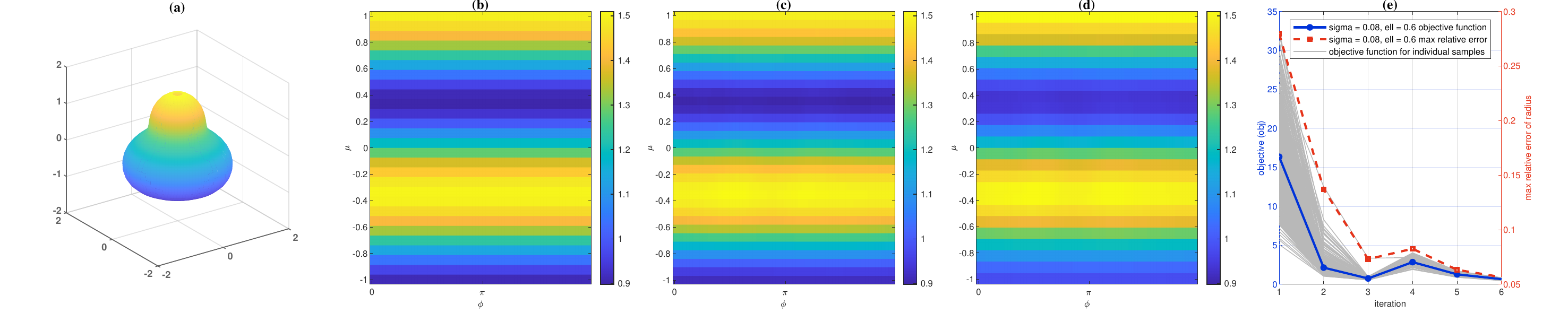}
    \end{overpic}
    \end{center}
    \begin{center}
        \begin{overpic}[width=0.9\textwidth,tics=10]{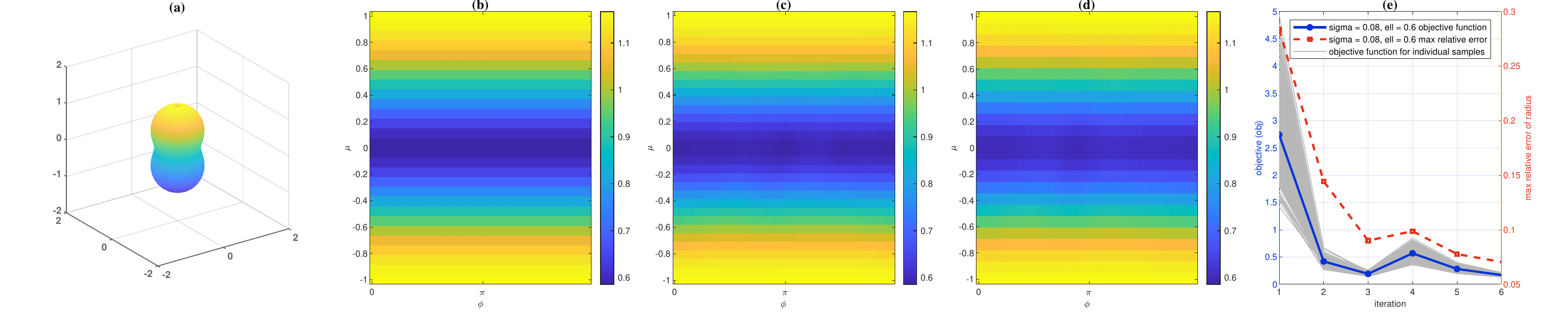}
      \end{overpic}
    \end{center}
\begin{center}
    \begin{overpic}[width=0.9\textwidth, tics=10]{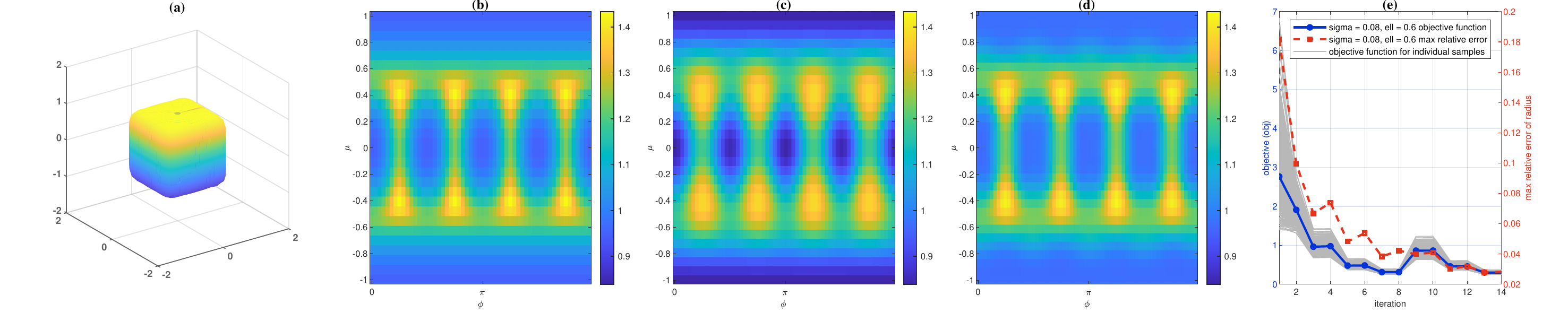}
    \end{overpic}
\end{center}
\vspace{0.6cm}
      \vspace{-0.2cm}
    \caption{ \scriptsize   (a): true geometry.
  (b): true base radius.
  (c): mean recovered radius  under case \uppercase\expandafter{\romannumeral1}: $\sigma$ = 0.06, $\ell_c$ = 0.8.
  (d): mean recovered radius  under case \uppercase\expandafter{\romannumeral2}: $\sigma$ = 0.08, $\ell_c$ = 0.6.
  (e): loss versus iteration under case \uppercase\expandafter{\romannumeral2}. blue solid line: objective function (left blue axis); red dashed line: max relative error of radius (right red axis); gray solid line: objective function for individual samples.}
    \label{example2}
      \end{figure}

  \begin{figure}
\begin{center}
    \begin{overpic}[width=0.6\textwidth, tics=10]{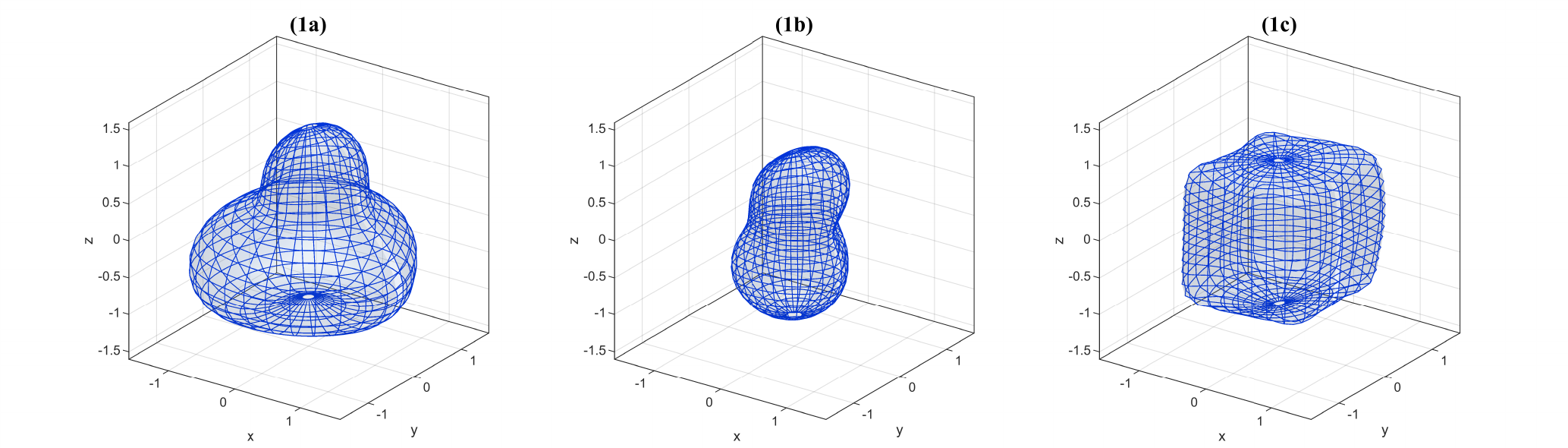}
    \end{overpic}
\end{center}
 \begin{center}
    \begin{overpic}[width=0.6\textwidth, tics=10]{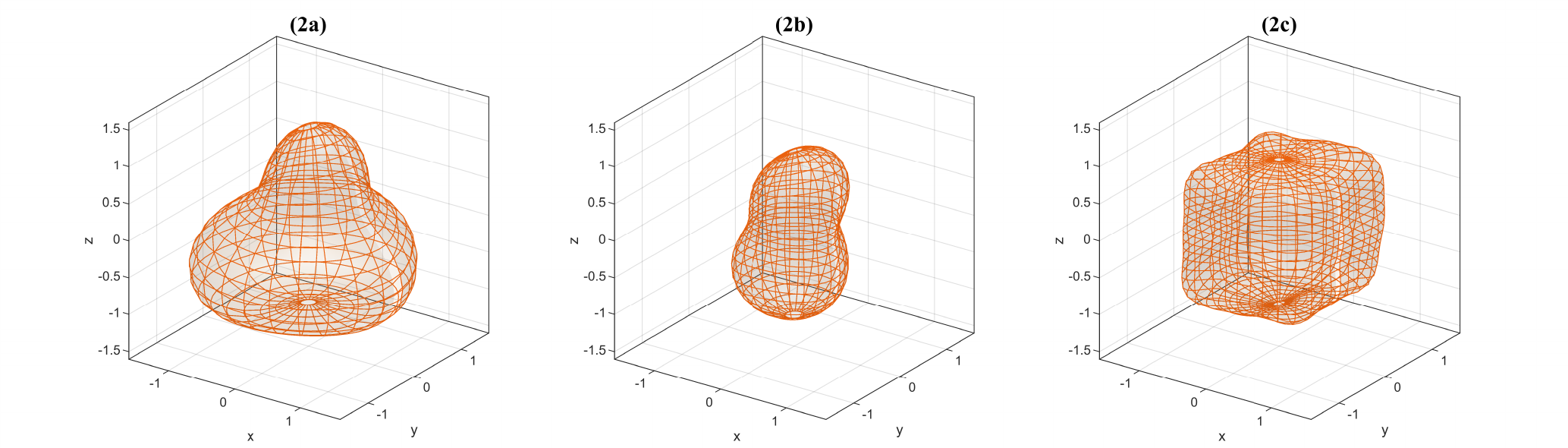}
    \end{overpic}
\end{center}
 \begin{center}
    \begin{overpic}[width=0.6\textwidth, tics=10]{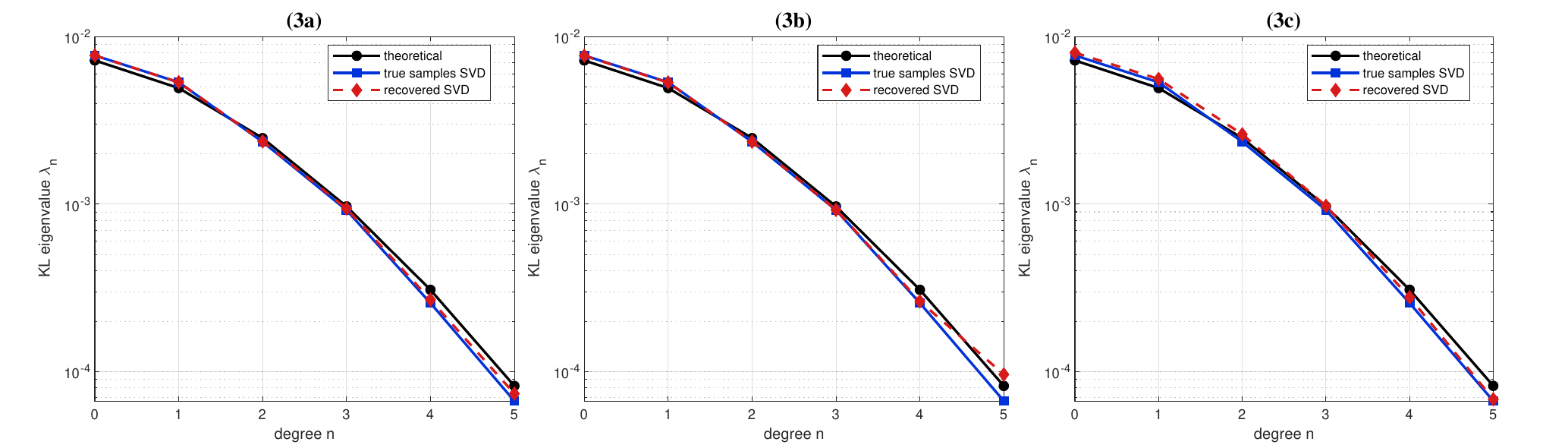}
    \end{overpic}
\end{center}
 \begin{center}
    \begin{overpic}[width=0.6\textwidth, tics=10]{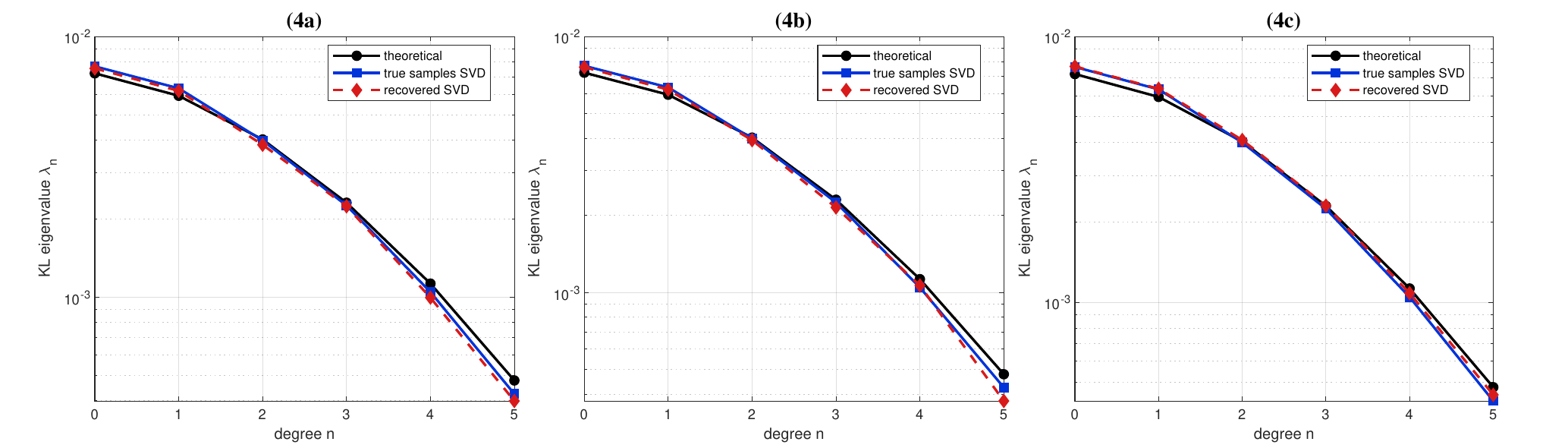}
    \end{overpic}
\end{center}
    \vspace{-0.3cm}
    \caption{\scriptsize  Comparison of inversion for random samples and reconstruction of KL eigenvalues. (1a): random pear sample (\#10), $\sigma$ = 0.06, $\ell_c$ = 0.8.
  (1b): random peanut sample (\#10),  $\sigma$ = 0.06, $\ell_c$ = 0.8.
  (1c): random rounded cube sample (\#10),  $\sigma$ = 0.06, $\ell_c$ = 0.8. (3a): pear, case \uppercase\expandafter{\romannumeral1}  KL comparison, $\sigma$ = 0.06, $\ell_c$ = 0.8, $M $= 200.   (2a): recovered pear shape, sampleId = 13.
  (2b): recovered peanut shape(\#10).
  (2c): recovered rounded cube shape(\#10).
  (3b): peanut, case \uppercase\expandafter{\romannumeral1}  KL comparison, $\sigma$ = 0.06, $\ell_c$ = 0.8, $M$ = 200.
  (3c): rounded cube, case \uppercase\expandafter{\romannumeral1}  KL comparison, $\sigma$ = 0.06, $\ell_c$ = 0.8, $M$ = 200. (4a): pear, case \uppercase\expandafter{\romannumeral2} KL comparison, $\sigma$= 0.08, $\ell_c$  = 0.6, $M $ = 500.
  (4b): peanut, case \uppercase\expandafter{\romannumeral2}  KL comparison, $\sigma$ = 0.08, $\ell_c$  = 0.6, $M $= 500.
  (4c): rounded cube, case \uppercase\expandafter{\romannumeral2}  KL comparison, $\sigma$ = 0.08, $\ell_c$  = 0.6, $M $ = 500.}
    \label{KLdi}
      \end{figure}
      
\begin{table}[htbp]
\centering
\tiny
\renewcommand{\arraystretch}{1.15}
\setlength{\tabcolsep}{3pt}

\begin{tabular}{llcccccccc}
\toprule
&& \multicolumn{4}{c}{case \uppercase\expandafter{\romannumeral1} ($\sigma=0.060$, $\ell_c=0.80$)}
& \multicolumn{4}{c}{case \uppercase\expandafter{\romannumeral2} ($\sigma=0.080$, $\ell_c=0.60$)} \\
\cmidrule(lr){3-6} \cmidrule(lr){7-10}
&& \shortstack{True\\\textup{(Preset)}}&  \shortstack{Reference\\\textup{}}& \shortstack{Estimate\\\textup{50 samples}} & \shortstack{Estimate\\100 samples}
& \shortstack{True\\\textup{(Preset)}} & \shortstack{Reference\\\textup{}}& \shortstack{Estimate\\\textup{200 samples}}& \shortstack{Estimate\\500 samples} \\
\midrule

\multirow{2}{*}{pear}
& $\sigma$  & 0.060 & 0.059 & 0.060 (100\%) & \textbf{0.060 (100\%)} & 0.080 & 0.078 & 0.076 (95\%) & \textbf{0.077 (99\%)} \\
& $\ell_c$  & 0.80  & 0.80  & 0.83 (96\%) & \textbf{0.80 (100\%)}  & 0.60  & 0.62  & 0.64 (93\%)& \textbf{0.62 (97\%)}  \\
\addlinespace

\multirow{2}{*}{peanut}
& $\sigma$  & 0.060 & 0.059  & 0.060 (100\%) & \textbf{0.060 (100\%)}  & 0.080 & 0.078 & 0.077 (96\%)& \textbf{0.077 (96\%)} \\
& $\ell_c$  & 0.80  & 0.80  & 0.78 (98\%)& \textbf{0.79 (99\%)}  & 0.60  & 0.62  & 0.65 (92\%) & \textbf{0.63 (95\%)}  \\
\addlinespace

\multirow{2}{*}{\shortstack{rounded\\square}}
& $\sigma$  & 0.060 & 0.059  & 0.062 (97\%)& \textbf{0.061 (98\%)}  & 0.080 & 0.078 & 0.079 (99\%)& \textbf{0.079 (99\%)} \\
& $\ell_c$  & 0.80  & 0.80  & 0.85 (94\%) & \textbf{0.81 (99\%)}  & 0.60  & 0.62  & 0.64 (93\%) & \textbf{0.61 (98\%)}  \\
\bottomrule
\end{tabular}

\caption{\scriptsize Estimated hyper-parameters for different base geometries.}
\label{tab:hyperparameter_geometries}
\end{table}

We again consider two parameter settings: one moderate (case \uppercase\expandafter{\romannumeral1}: $\sigma=0.06$, $\ell_c=0.8$) and one rougher (case \uppercase\expandafter{\romannumeral2}: $\sigma=0.08$, $\ell_c=0.6$). \autoref{exam2} displays random realizations of these scatterers, along with comparisons between perturbed samples and the true geometry under both parameter regimes. In this part the wavenumber is still set to be $k=[3,4,5,6]$. For  case \uppercase\expandafter{\romannumeral1} we compare the inversion results using 50 and 100 samples; for the more irregular parameter case \uppercase\expandafter{\romannumeral2}, we compare the results using 200 and 500 samples instead. \autoref{example2}  displays the comparison between the reconstructed mean shape and the true geometry, together with the objective function's evolution over iterations.
\autoref{KLdi} presents the inversion outcomes of KL eigenvalues for distinct shapes under the two parameter regimes. A summary of the hyper-parameter statistics is also provided in \autoref{tab:hyperparameter_geometries}. Based on the present inversion results, the outcomes remain reasonably satisfactory.

   \begin{figure}
\begin{center}
    \begin{overpic}[width=0.9\textwidth, tics=10]{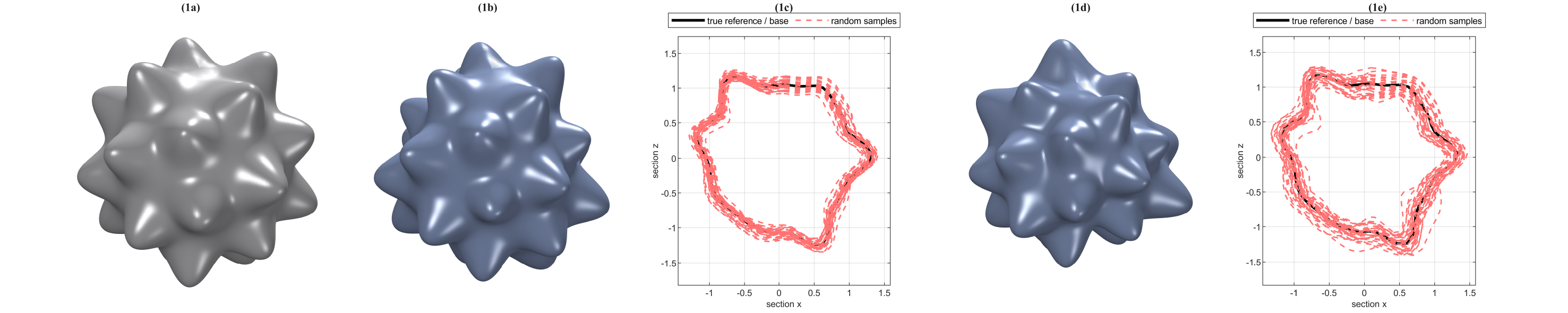}
    \end{overpic}
\end{center}
  \vspace{0.3cm}
\begin{center}
    \begin{overpic}[width=0.9\textwidth, tics=10]{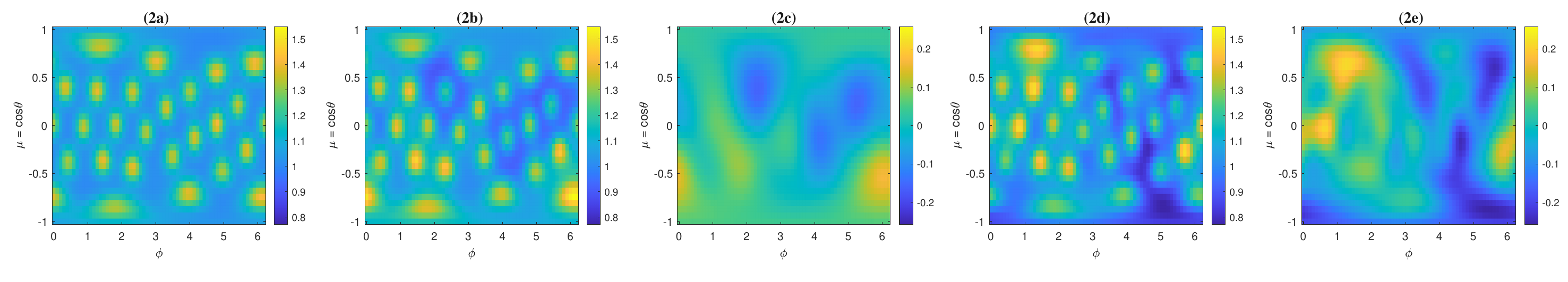}
    \end{overpic}
\end{center}
 \vspace{-0.5cm}
    \caption{\scriptsize Geometry visualization. (1a): reference/base scatterer.
  (1b): selected random scatterer under case \uppercase\expandafter{\romannumeral1}(\#10).
  (1c): section overlay under case  \uppercase\expandafter{\romannumeral1}.
  (1d): selected random scatterer under case  \uppercase\expandafter{\romannumeral2}(\#10).
  (1e): section overlay under case  \uppercase\expandafter{\romannumeral2}. (2a): reference/base radius.
(2b): selected random radius under case  \uppercase\expandafter{\romannumeral1}(\#10).
(2c): random perturbation under case  \uppercase\expandafter{\romannumeral1}.
(2d): selected random radius under case  \uppercase\expandafter{\romannumeral2}(\#10).
(2e): random perturbation under case  \uppercase\expandafter{\romannumeral2}.}
    \label{pic:complexshape}
      \end{figure}

\subsection{Example 3: role of the multi-frequency continuation strategy}
\label{sec:con}
Through this example, we demonstrate that the frequency-continuation strategy plays an important role in inverse scattering for random obstacles since it improves the recovery of both the obstacle geometry and the statistical quantities of the random field. Consider a more complex geometric  three-dimensional scatterer whose reference shape is depicted in \autoref{pic:complexshape}(a). \autoref{pic:complexshape}(c) and \autoref{pic:complexshape}(e) compare the randomly perturbed surfaces with the reference shape under the two parameter regimes. The inversion results obtained with the multi-frequency continuation framework are presented in \autoref{pic:complexshapere} where eight consecutive wavenumbers $(k = 1, 2, …, 8)$ are used for the inversion.  \autoref{fresingco} further illustrates how the reconstructed shape evolves with the frequency. As the frequency increases, finer geometric details are progressively recovered.

   \begin{figure}
\begin{center}
    \begin{overpic}[width=0.8\textwidth, tics=10]{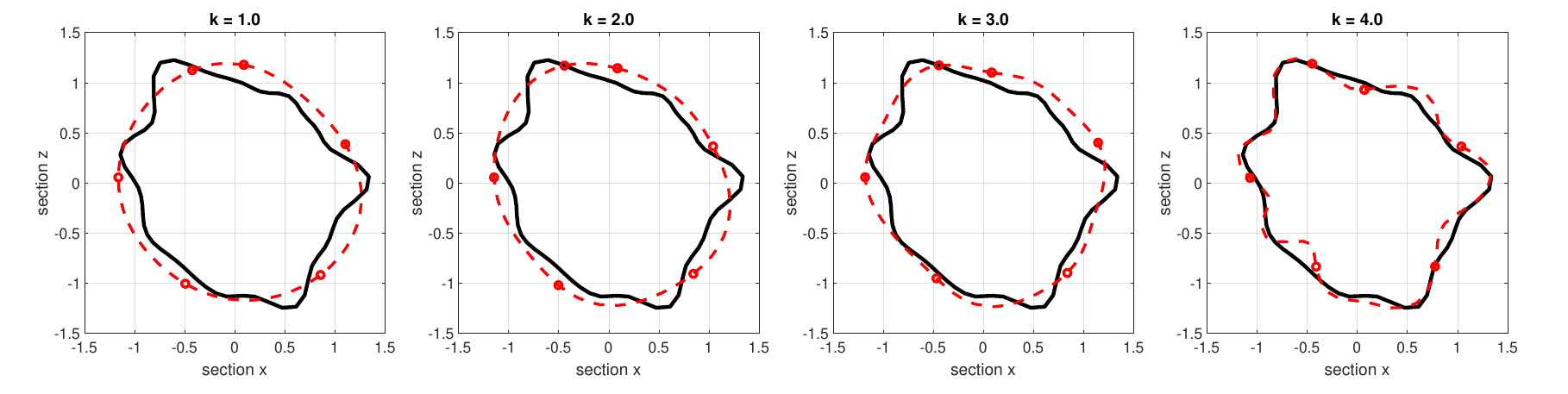}
    \end{overpic}
\end{center}
 \begin{center}
    \begin{overpic}[width=0.8\textwidth, tics=10]{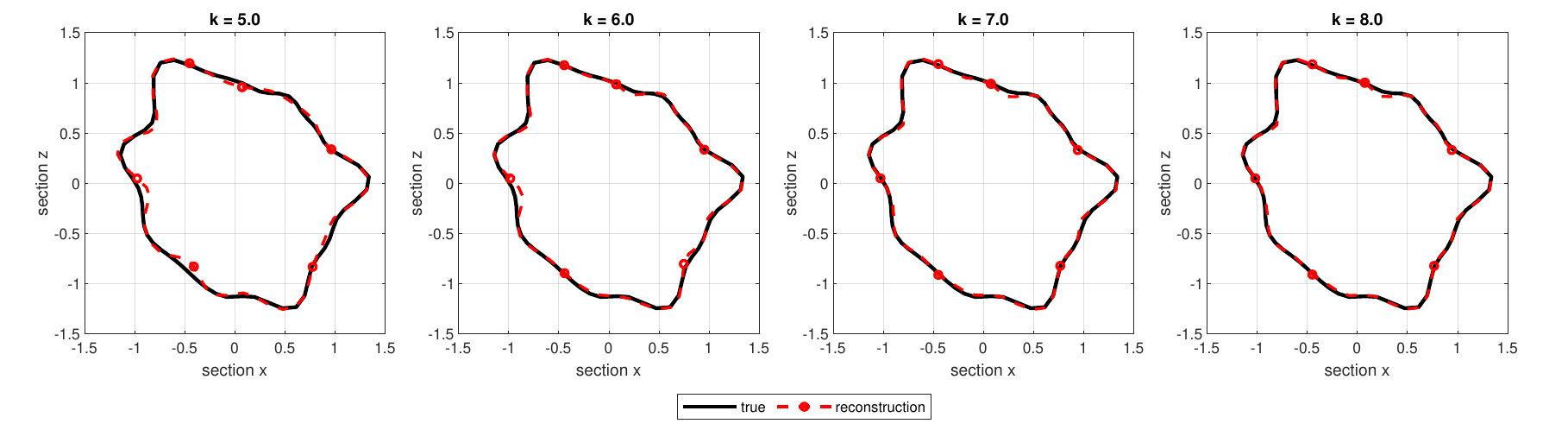}
    \end{overpic}
\end{center}
      \vspace{-0.4cm}
    \caption{\scriptsize Shape evolution during multi-frequency continuation.}
    \label{fresingco}
      \end{figure}

 \begin{figure}
    \begin{center}
    \begin{overpic}[width=0.9\textwidth, tics=10]{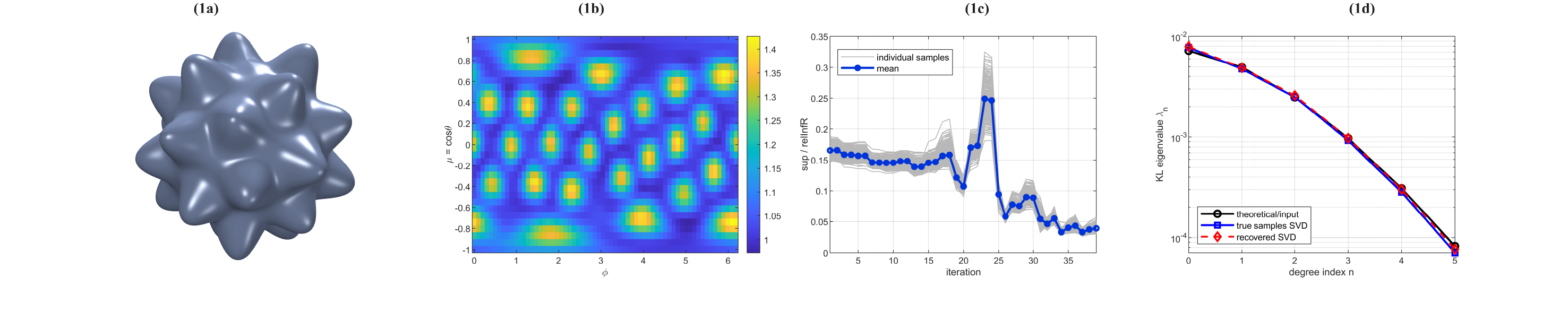}
    \end{overpic}
    \end{center}
    \vspace{-0.2cm}
    \begin{center}
        \begin{overpic}[width=0.9\textwidth, tics=10]{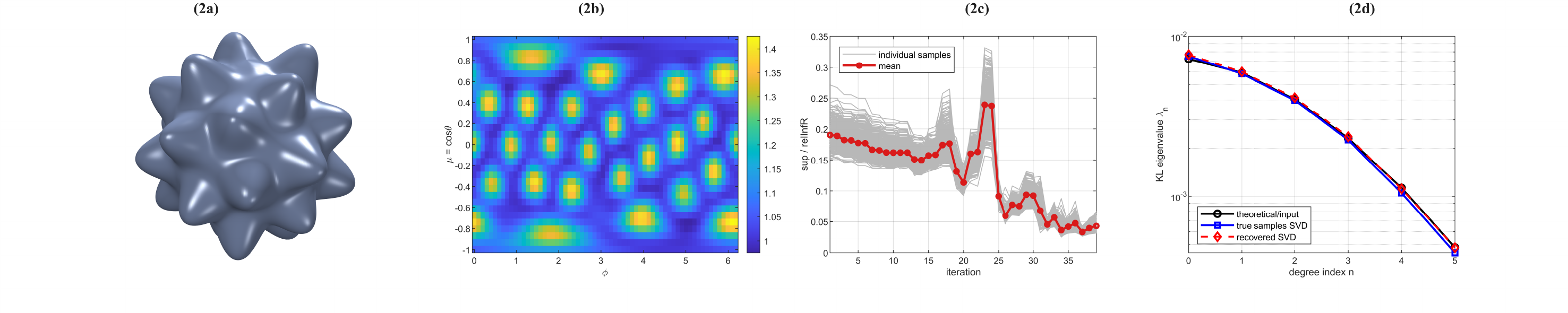}
      \end{overpic}
    \end{center}
    \vspace{-0.4cm}
    \caption{ \scriptsize Reconstructions results under multi-frequency scheme. (1a): mean recovered shape, $M$ = 300.
  (1b): mean recovered radius, $\sigma=0.06$, $\ell_c=0.8$.
  (1c): sup-norm radius error history; gray solid line: individual samples, colored dashed curve: mean.
  (1d):  KL eigenvalues comparison; black solid line: theoretical/input, blue solid line: true samples, red solid line: recovered.  (2a): mean recovered 3D shape, $M = 600$; 
(2b): mean recovered radius heatmap, $\sigma=0.08$, $\ell_c=0.6$; 
(2c): sup-norm radius error history; gray solid line: individual samples, colored dashed curve: mean; 
(2d):  KL eigenvalue comparison; black solid line: theoretical/input, blue solid line: true samples, red solid line: recovered.}
    \label{pic:complexshapere}
      \end{figure}

   \begin{table}[htbp]
\centering
\tiny
\renewcommand{\arraystretch}{1.18}
\setlength{\tabcolsep}{4pt}
\setlength{\heavyrulewidth}{1.05pt}
\setlength{\lightrulewidth}{0.55pt}
\setlength{\cmidrulewidth}{0.55pt}

\begin{tabularx}{\textwidth}{
ll
>{\centering\arraybackslash}p{1.7cm}
>{\centering\arraybackslash}p{1.5cm}
>{\centering\arraybackslash}X
>{\centering\arraybackslash}X
>{\centering\arraybackslash}X
}
\toprule
&& True(Preset)
& Reference
& \makecell{Estimated(20 samples)}
& \makecell{Estimated(50 samples)}
& \makecell{Estimated(100 samples)} \\
\midrule

\multirow{2}{*}{case \uppercase\expandafter{\romannumeral1}}
& $\sigma$  & 0.0600 & 0.0604 & 0.063 (95\%) & 0.0614 (98\%)
& \textbf{ 0.0604 (99\%)} \\
& $\ell_c$ & 0.80   & 0.79  & 0.96 (80\%)& 0.83 (96\%)  & \textbf{0.80  (100\%)} \\
\midrule

&& True(Preset)
& Reference
& \makecell{Estimated(200 samples)}
& \makecell{Estimated(400 samples)}
& \makecell{Estimated(600 samples)} \\
\midrule

\multirow{2}{*}{case \uppercase\expandafter{\romannumeral2}}
& $\sigma$  & 0.0800 & 0.0773 & 0.0798 (99\%) & 0.0792 (99\%)& \textbf{0.0789 (99\%)} \\
& $\ell_c$ & 0.60   & 0.61   & 0.63  (95\%) & 0.62 (97\%)  & \textbf{0.61 (99\%)} \\
\bottomrule
\end{tabularx}

\caption{\scriptsize  Comparison of estimated hyper-parameters under two settings with different sample sizes. True (Preset) denotes the ground-truth covariance parameters, while Reference denotes estimates recovered via inversion with the true geometric configuration.}
\label{tab:hyperparameter_estimates}
\end{table}











\begin{figure}
\centering

\begin{tikzpicture}


\node[inner sep=0pt, anchor=north west] (r1) at (0,0) {%
  \begin{overpic}[width=0.9\textwidth, tics=10]{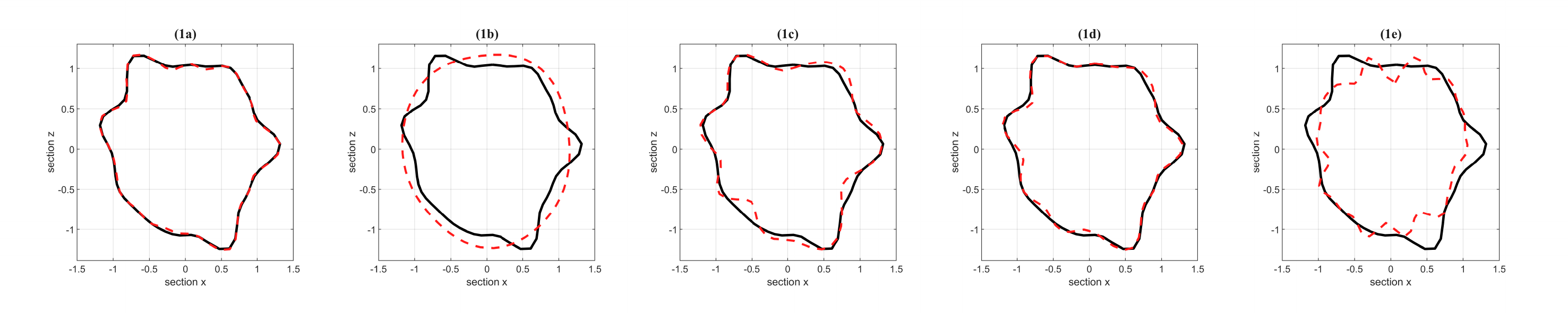}
  \end{overpic}
};

\node[inner sep=0pt, anchor=north west] (r2) at ($(r1.south west)+(0,-0.35cm)$) {%
  \begin{overpic}[width=0.9\textwidth, tics=10]{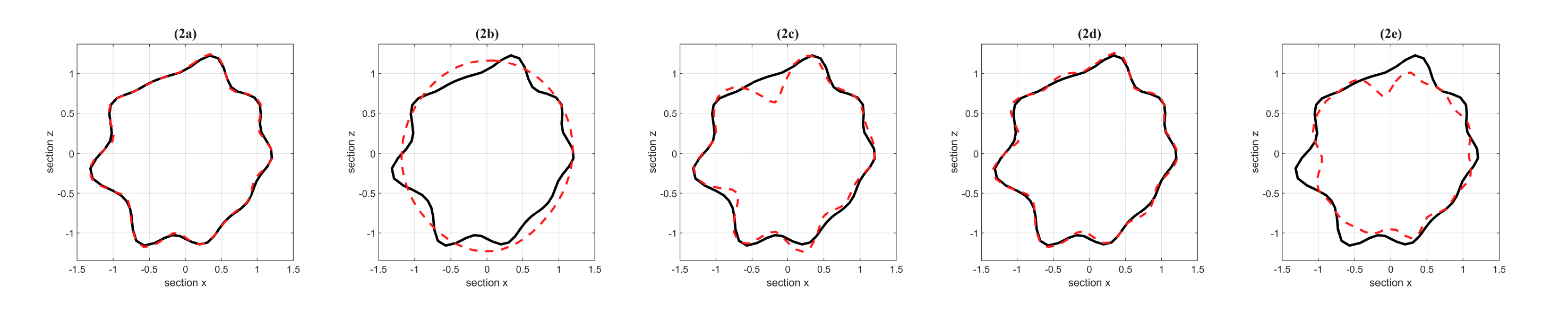}
  \end{overpic}
};

\node[inner sep=0pt, anchor=north west] (r3) at ($(r2.south west)+(0,-0.35cm)$) {%
  \begin{overpic}[width=0.9\textwidth, tics=10]{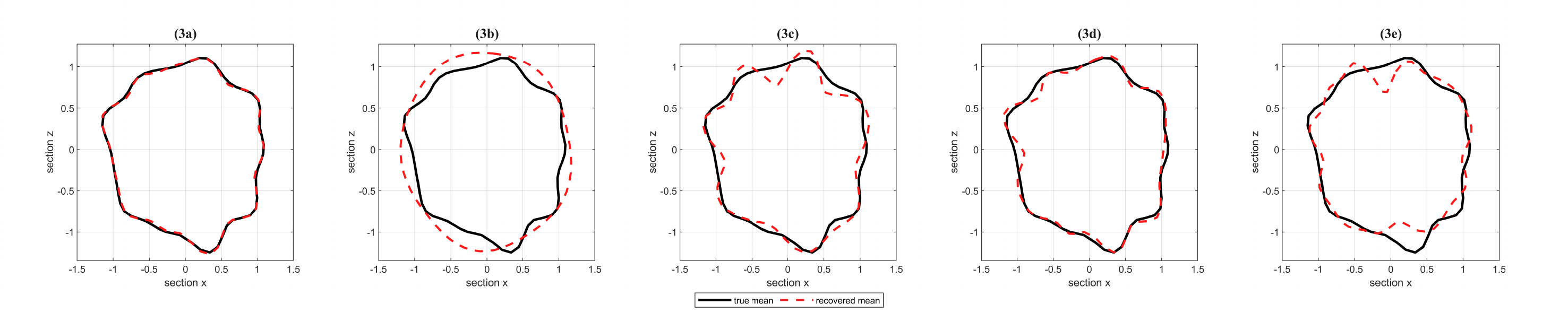}
  \end{overpic}
};


\pgfmathsetmacro{\xc}{0.110}
\pgfmathsetmacro{\xa}{\xc-0.088}
\pgfmathsetmacro{\xb}{\xc+0.090}

\draw[
  black,
  line width=0.65pt,
  dash pattern=on 3pt off 2pt,
  rounded corners=1.5pt
]
  ($(r1.north west)!\xa!(r1.north east)+(0,-0.04cm)$)
  rectangle
  ($(r3.south west)!\xb!(r3.south east)+(0,0.04cm)$);


\foreach \xc in {0.305,0.495,0.685,0.875}{
  \pgfmathsetmacro{\xa}{\xc-0.088}
  \pgfmathsetmacro{\xb}{\xc+0.090}

  \draw[
    black!30,
    line width=0.55pt,
    dash pattern=on 3pt off 2pt,
    rounded corners=1.5pt
  ]
    ($(r1.north west)!\xa!(r1.north east)+(0,-0.04cm)$)
    rectangle
    ($(r3.south west)!\xb!(r3.south east)+(0,0.04cm)$);
}

\end{tikzpicture}

\caption{\scriptsize Comparison of true mean vs recovered mean for multi-frequency and single-frequency ($k=2,4,6,8$) inversions under case \uppercase\expandafter{\romannumeral1}, each with 300 samples. (1a)--(3a): multi-frequency results, sections $\phi = 0, 1.0472, 2.0944$ respectively; (1b)--(3b): $k = 2$ results, sections $\phi = 0, 1.0472, 2.0944$ respectively; (1c)--(3c): $k = 4$ results, sections $\phi = 0, 1.0472, 2.0944$ respectively; (1d)--(3d): $k = 6$ results, sections $\phi = 0, 1.0472, 2.0944$ respectively; (1e)--(3e): $k = 8$ results, sections $\phi = 0, 1.0472, 2.0944$ respectively. Each dashed rectangle groups three cross-sectional comparisons of the same mean shape under one inversion setting.}
\label{frecom}

\end{figure}
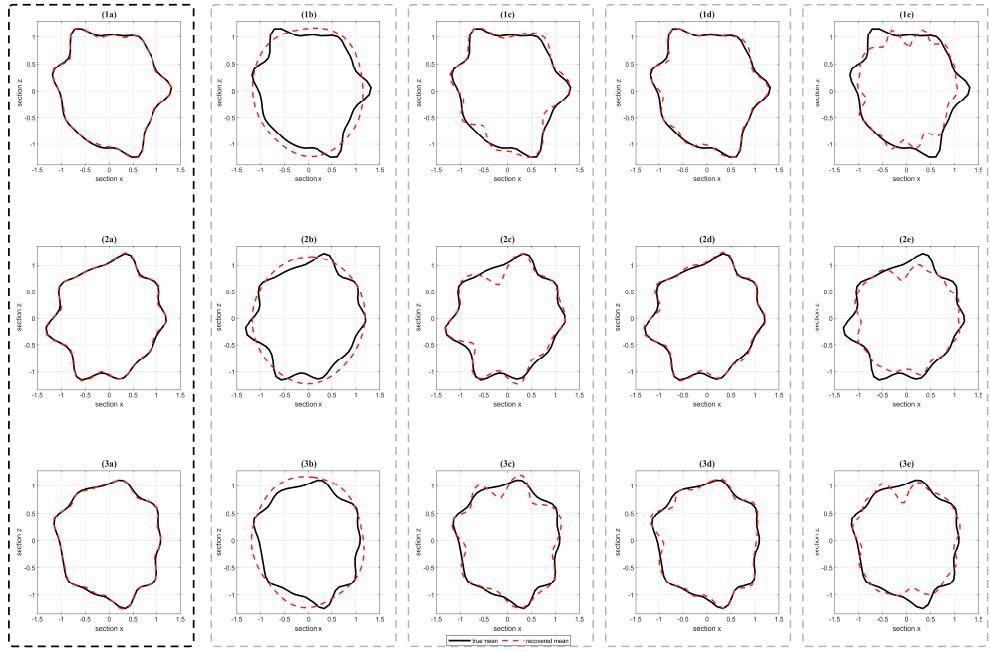

\autoref{pic:complexshapere}(b) shows the sample-averaged reconstructed radius for case~\uppercase\expandafter{\romannumeral2} and the corresponding hyper-parameter estimates are reported in \autoref{tab:hyperparameter_estimates}. Compared with case~\uppercase\expandafter{\romannumeral1}, the larger fluctuation amplitude and shorter correlation length make the inversion more challenging, and a larger number of samples is needed to maintain reliable accuracy in the reconstructed mean shape, KL spectrum and covariance hyper-parameters.

\autoref{pic:complexshapere}(c) shows the evolution of the relative $L^\infty$ error during the iteration. Several transient increases can be observed, mainly caused by local shape adjustments when the frequency is switched or the regularization strength is reduced. Since the $L^\infty$ error is sensitive to maximum pointwise deviations, such local overshoots are amplified in the error curve. Nevertheless, these increases are temporary, and the overall error trend remains decreasing.

To assess the role of frequency continuation, we compare the proposed multi-frequency strategy with single-frequency inversions. Four individual wavenumbers ($k=2,4,6,8$) are tested separately, while all other parameters are kept the same as in the previous experiment. In each single-frequency test, the reconstruction procedure in \autoref{alg:MC-MFRLA-mean} is applied directly at the prescribed wavenumber without using lower-frequency reconstructions as initial guesses. The comparison is performed with 300 inversion samples under case~\uppercase\expandafter{\romannumeral1} and the cross-sectional profiles of the sample-averaged reconstructions are shown in \autoref{frecom}.

From \autoref{frecom}, the single-frequency reconstructions exhibit clear limitations compared with the multi-frequency continuation result. Low-frequency data provide relatively stable recovery of the global shape but fail to resolve fine geometric features, whereas high-frequency data contain more local information but are more sensitive to the initial guess, noise and instability. Consequently, the reconstructed mean shapes are less accurate than those obtained by the continuation strategy. This comparison shows that the multi-frequency continuation scheme provides a more reliable balance between stability and resolution for complex random scatterers.

 \begin{figure}
\begin{center}
    \begin{overpic}[width=0.8\textwidth, tics=10]{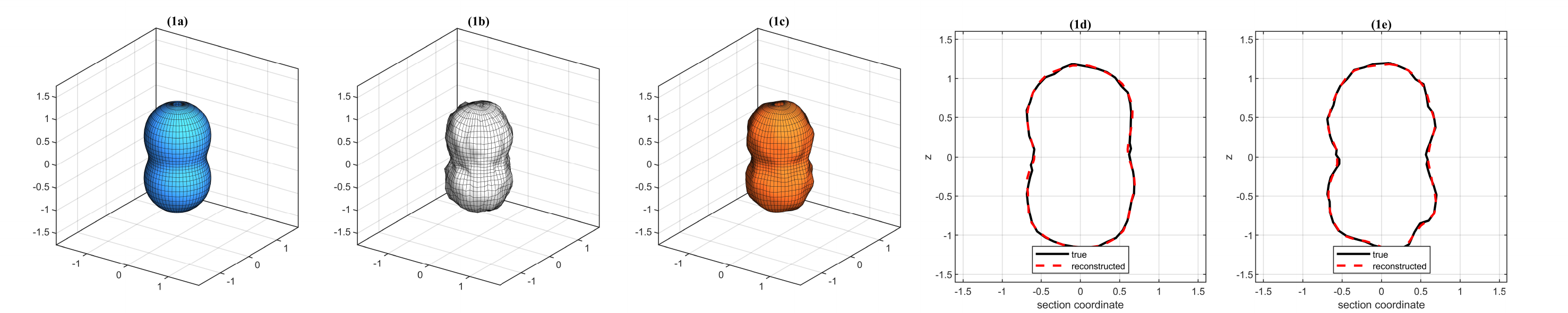}
    \end{overpic}
\end{center}
 \begin{center}
    \begin{overpic}[width=0.8\textwidth, tics=10]{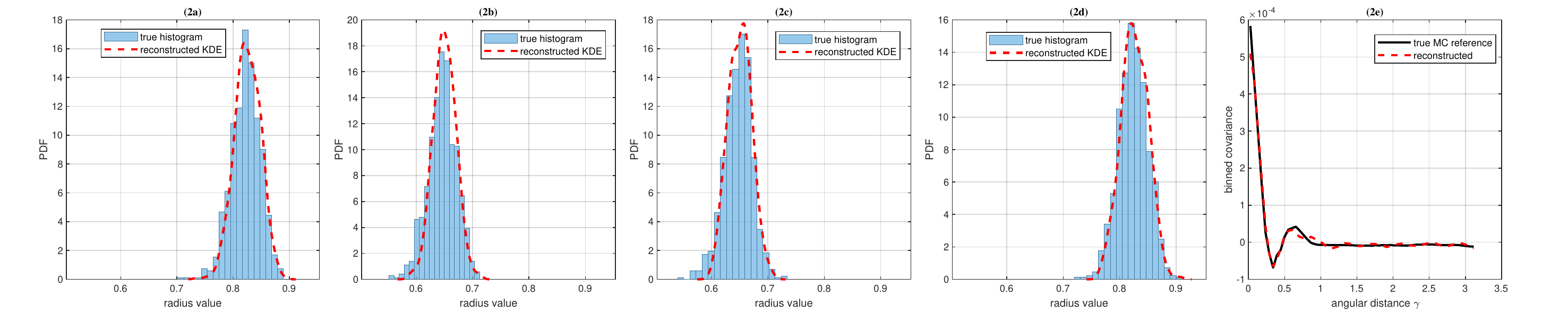}
    \end{overpic}
\end{center}
 \begin{center}
    \begin{overpic}[width=0.8\textwidth, tics=10]{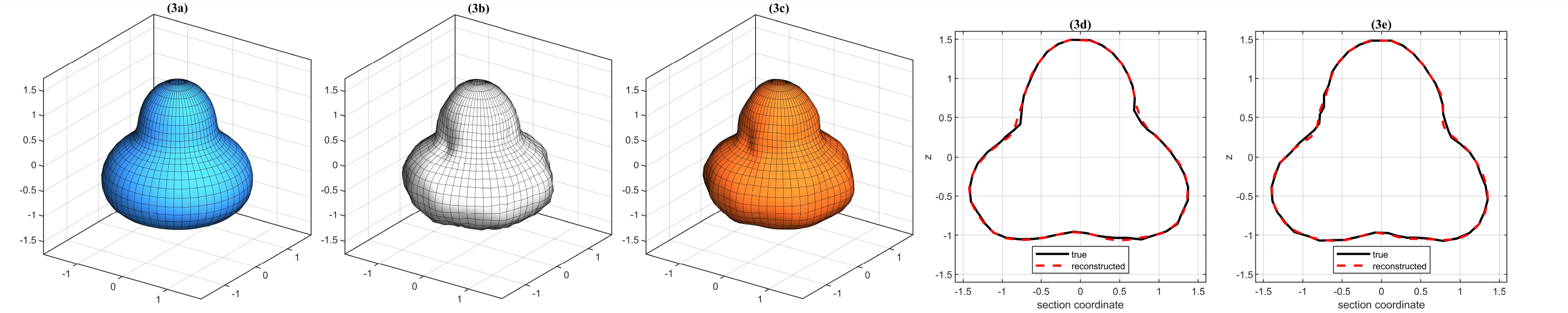}
    \end{overpic}
\end{center}
 \begin{center}
    \begin{overpic}[width=0.8\textwidth, tics=10]{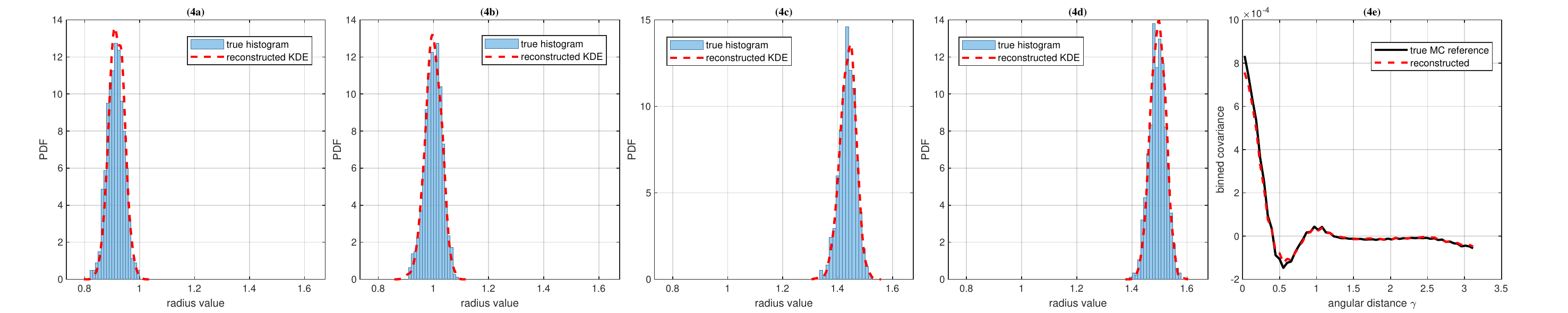}
    \end{overpic}
\end{center}
    \vspace{-0.3cm}
    \caption{\scriptsize Comparison of inversion results under a non-Gaussian random field. (1a)(3a): mean reconstructed shape, $M = 800$. (1b)(3b): true random sample$(\#1)$.
(1c)(3c): reconstructed sample$(\#1)$.
(1d)(3d): section comparison at $\phi = 0$; black solid line: true shape. red dashed line: reconstructed shape.
(1e)(3e): section comparison at $\phi = \pi/2$; black solid line: true shape. red dashed line: reconstructed shape.
(2a)(4a):  PDF of radius at grid $(\mu, \phi) = (0.57, 0.79)$; blue histogram: true radius samples. red dashed line: reconstructed radius KDE.
(2b)(4b): PDF of radius at grid $(\mu, \phi) = (0.27, 2.36)$; blue histogram: true radius samples. red dashed line:  reconstructed radius KDE.
(2c)(4c):  PDF of radius at grid $(\mu, \phi) = (-0.27, 3.93)$; blue histogram: true radius samples. red dashed line: reconstructed radius KDE.
(2d)(4d):  PDF of radius at grid $(\mu, \phi) = (-0.57, 5.50)$; blue histogram: true radius samples. red dashed line: reconstructed radius KDE.
(2e)(4e): covariance. black solid line: true reference covariance, red dashed line: reconstructed covariance.}
    \label{nonr}
      \end{figure}

      \begin{figure}
\begin{center}
    \begin{overpic}[width=0.8\textwidth, tics=10]{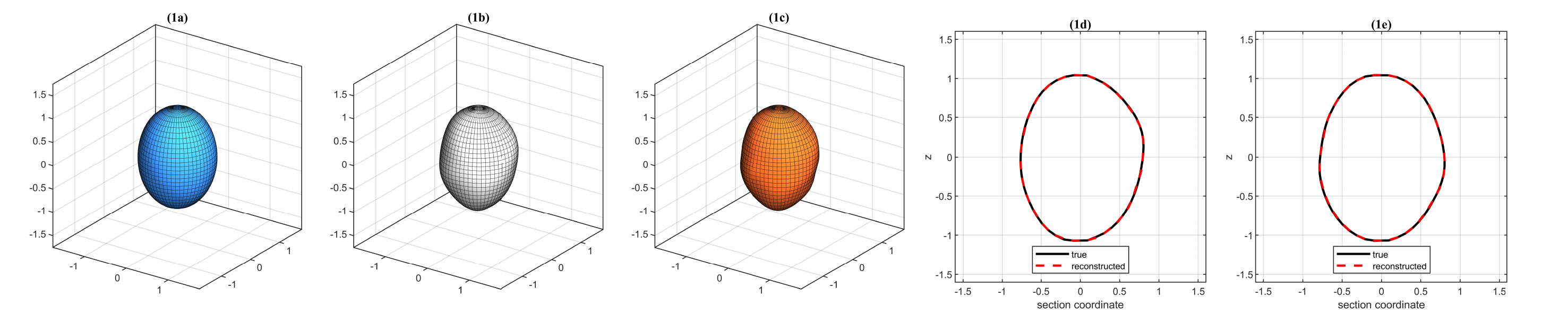}
    \end{overpic}
\end{center}
 \begin{center}
    \begin{overpic}[width=0.8\textwidth, tics=10]{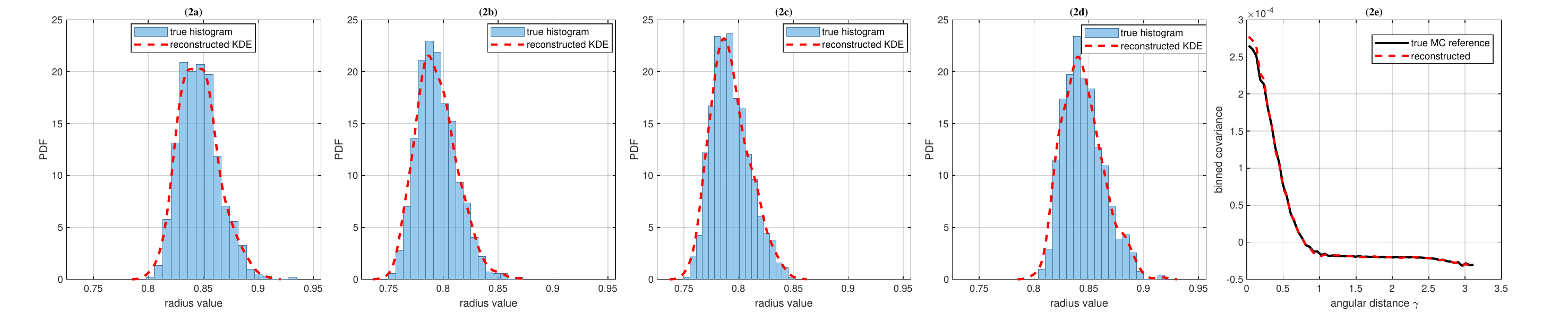}
    \end{overpic}
\end{center}
  \vspace{-0.3cm}
    \caption{\scriptsize Comparison of inversion results for a strongly non-Gaussian random field, characterized by spatial mean absolute skewness and excess kurtosis values of 0.6459 and 0.4725, respectively. (1a): mean reconstructed shape, $M = 800$. (1b): true random sample$(\#1)$.
(1c): reconstructed sample$(\#1)$.
(1d): section comparison at $\phi = 0$; black solid line: true shape. red dashed line: reconstructed shape.
(1e): section comparison at $\phi = \pi/2$; black solid line: true shape. red dashed line: reconstructed shape.
(2a):  PDF of radius at grid $(\mu, \phi) = (0.57, 0.79)$; blue histogram: true radius samples. red dashed line: reconstructed radius KDE.
(2b): PDF of radius at grid $(\mu, \phi) = (0.27, 2.36)$; blue histogram: true radius samples. red dashed line:  reconstructed radius KDE.
(2c):  PDF of radius at grid $(\mu, \phi) = (-0.27, 3.93)$; blue histogram: true radius samples. red dashed line: reconstructed radius KDE.
(2d):  PDF of radius at grid $(\mu, \phi) = (-0.57, 5.50)$; blue histogram: true radius samples. red dashed line: reconstructed radius KDE.
(2e): angular-binned covariance. black solid line: true reference covariance, red dashed line: reconstructed covariance.}
    \label{skeness}
      \end{figure}

      \begin{figure}
    \centering
       \includegraphics[width=\linewidth]{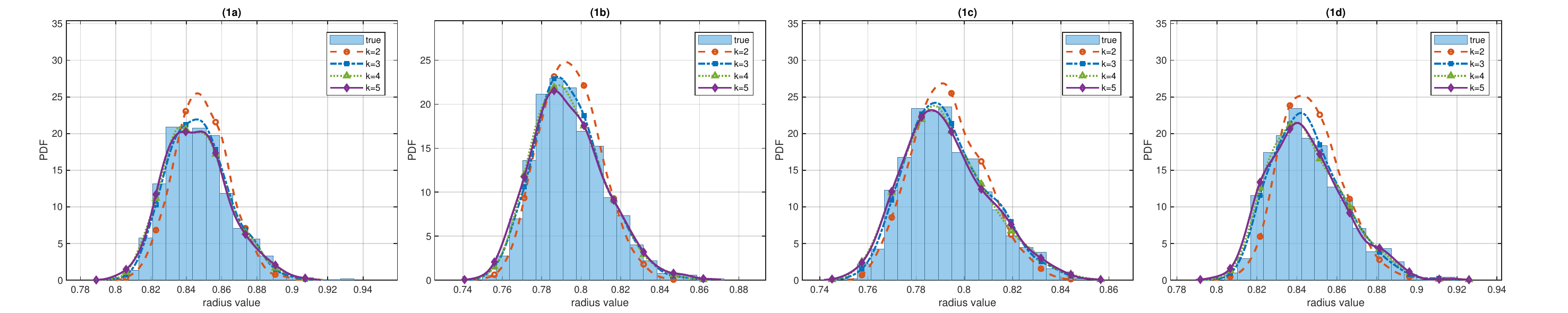}
        \caption{\scriptsize  Pointwise radius PDFs at four selected directions $(\mu,\phi)$=$(0.6,\pi/4)$, $(0.3,3\pi/4)$, $(-0.3,5\pi/4)$, $(-0.6,7\pi/4)$ for the non-Gaussian ellipsoidal case. The blue histograms and black solid curves show the true samples, and the colored dashed curves show the reconstructed KDEs for different end wavenumbers.}
        \label{ellppp}
    \end{figure}

\subsection{Example 4: non-Gaussian random field}

In this example, we consider random perturbations generated from non-Gaussian random fields. The samples are produced by a digital-filter method \cite{wu1988non} where spatial correlation is introduced through a spherical filter depending on the angular distance; thus, the construction is isotropic by design. In practice, independent non-Gaussian random variables are first generated on a discrete spherical grid and then filtered to obtain spatially correlated perturbations. Two types of non-Gaussian random fields are tested, each with 800 samples, and the inversion results are shown in \autoref{nonr}. The multi-frequency reconstruction is performed using four consecutive wavenumbers $k=2,3,4,5$. The angle-averaged covariance curves shown in \autoref{nonr}(2e) and \autoref{nonr}(4e) are computed by first estimating empirical covariances between pairs of surface directions and then averaging over all direction pairs with the same angular distance. The close agreement with the reference curve, shown as the black solid line, indicates that the reconstructed samples recover the individual shapes accurately and capture the spatial correlation features of the underlying non-Gaussian random field.

  \begin{figure}
    \centering
       \includegraphics[width=0.9\linewidth]{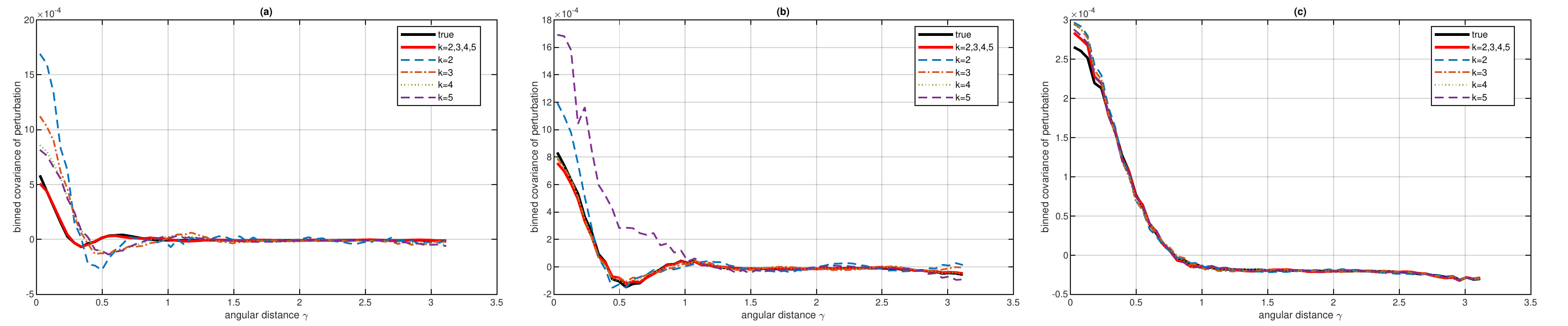}
       \caption{\scriptsize Comparison of  covariance functions reconstructed from multi-frequency and single-frequency far-field data for three non-Gaussian random scatterers. Panels (a)-(c) correspond to the peanut-shaped, pear-shaped, and vertical ellipsoidal scatterers, respectively. The black solid curves denote the reference covariance functions computed from the true samples. The red solid curves correspond to the multi-frequency reconstruction using $k=2,3,4,5$ while the dashed colored curves show the covariance estimates obtained from single-frequency reconstructions at $k=2,3,4,5$.}
        \label{ellk}
    \end{figure}

    To examine the recovery of local distributional features, four representative points are selected on the spherical parameter domain. At each point, the true and reconstructed perturbation values are collected over all samples, and kernel density estimation (KDE) is used to compare their one-dimensional marginal probability densities. The resulting density curves show that the reconstructed samples capture the main marginal distributional characteristics of the underlying non-Gaussian random field at the selected locations.

\begin{table}[htbp]
\centering
\tiny
\renewcommand{\arraystretch}{1.18}
\setlength{\tabcolsep}{4pt}
\setlength{\heavyrulewidth}{1.05pt}
\setlength{\lightrulewidth}{0.55pt}
\setlength{\cmidrulewidth}{0.55pt}

\begin{tabularx}{\textwidth}{
ll
>{\centering\arraybackslash}X
>{\centering\arraybackslash}X
>{\centering\arraybackslash}X
>{\centering\arraybackslash}X
>{\centering\arraybackslash}X
>{\centering\arraybackslash}X
>{\centering\arraybackslash}X
}
\toprule
Shape & Relative error
& $N=10$ & $N=50$ & $N=100$ & $N=200$ & $N=400$ & $N=600$ & $N=800$ \\
\midrule

\multirow{2}{*}{peanut}
& $E_{\rm mean}$ (\%)
& 0.90 & 0.54 & 0.46 & 0.43 & 0.41 & 0.41 & 0.40 \\
& $E_{\rm cov}$ (\%)
& 20.53 & 13.23 & 11.77 & 11.38 & 11.24 & 11.19 & 11.17 \\

\midrule

\multirow{2}{*}{pear}
& $E_{\rm mean}$ (\%)
& 0.75 & 0.41 & 0.34 & 0.29 & 0.27 & 0.26 & 0.26 \\
& $E_{\rm cov}$ (\%)
& 19.47 & 10.85 & 8.48 & 8.62 & 8.38 & 8.21 & 8.21 \\

\midrule

\multirow{2}{*}{vertical ellipsoid}
& $E_{\rm mean}$ (\%)
& 0.70 & 0.31 & 0.22 & 0.15 & 0.10 & 0.08 & 0.07 \\
& $E_{\rm cov}$ (\%)
& 18.35 & 8.53 & 6.95 & 6.46 & 5.38 & 5.13 & 5.01 \\

\bottomrule
\end{tabularx}

\caption{\scriptsize Sample-size dependence of the statistical reconstruction errors for three non-Gaussian random scatterers. Here $E_{\rm mean}$ denotes the relative $L^2$ error of the reconstructed sample mean radius, and $E_{\rm cov}$ denotes the relative $L^2$ error of the angle-averaged covariance function. All errors are reported in percentage form. For each sample size $N<800$, the values are averaged over 30 random subsamplings; for $N=800$, the full reconstructed sample set is used.}
\label{tab:sample_size_errors_three_shapes}
\end{table}

Next, we consider a more strongly non-Gaussian perturbation model. Its spatial mean absolute skewness and excess kurtosis are 0.6459 and 0.4725, respectively, substantially exceeding those of the previously considered peanut-shaped case (0.2193,0.1817) and pear-shaped case (0.1223,0.1679). This example examines the recovery of distributional features of the random radius beyond its mean geometry and covariance structure. The recovery results are shown in \autoref{skeness}. Despite the stronger non-Gaussianity, the proposed method provides stable reconstructions of the marginal radius distributions at the selected surface directions.

The effect of sample size is further investigated by computing the relative errors of the reconstructed sample mean and covariance function, as reported in \autoref{tab:sample_size_errors_three_shapes}. For all three non-Gaussian examples, the errors decrease rapidly for small sample sizes and then level off as the sample size becomes moderate. This indicates that a sufficiently large reconstructed samples can provide reliable estimates of both the mean geometry and the spatial correlation structure of the random shape fluctuation.








\begin{table}[htbp]
\centering
\tiny
\renewcommand{\arraystretch}{1.18}
\setlength{\tabcolsep}{3pt}
\setlength{\heavyrulewidth}{1.05pt}
\setlength{\lightrulewidth}{0.55pt}
\setlength{\cmidrulewidth}{0.55pt}

\newcolumntype{C}{>{\centering\arraybackslash}X}

\begin{tabularx}{\textwidth}{
>{\centering\arraybackslash}p{2.4cm}
C C C C C
}
\toprule
Shape
& \makecell{Single-frequency\\$k=2$}
& \makecell{Single-frequency\\$k=3$}
& \makecell{Single-frequency\\$k=4$}
& \makecell{Single-frequency\\$k=5$}
& \makecell{Multi-frequency\\$k=2,3,4,5$} \\
\midrule

peanut
& 2.6805
& 1.5601
& 1.1459
& 1.0934
& \textbf{0.1117} \\

pear
& 0.4724
& 0.1066
& 0.0637
& 1.5569
& \textbf{0.0821} \\

vertical ellipsoid
& 0.0981
& 0.0847
& 0.0819
& 0.0612
& \textbf{0.0501} \\

\bottomrule
\end{tabularx}

\caption{\scriptsize Relative $L^2$ errors of the covariance functions reconstructed from
single-frequency and multi-frequency far-field data for three non-Gaussian random
scatterers. The error is computed by
$\|C_{\rm rec}-C_{\rm true}\|_2/\|C_{\rm true}\|_2$. Smaller values indicate
better agreement with the reference covariance function.}
\label{tab:covariance_error_single_multi}
\end{table}

To further examine the role of frequency continuation in this more challenging non-Gaussian setting, we compare the marginal radius PDFs reconstructed with different terminal wavenumbers. The results are shown in \autoref{ellppp}, where the blue histograms represent the true empirical distributions and the colored dashed curves show the KDE-based density estimates from the reconstructed samples. As the terminal wavenumber increases, the reconstructed PDFs approach the true marginal distributions, indicating improved recovery of non-Gaussian distributional features.

We also compare the continuation strategy with single-frequency reconstructions at $k=2,3,4,5$. The reconstructed covariance functions are shown in \autoref{ellk}, and the corresponding relative errors are reported in \autoref{tab:covariance_error_single_multi}. Although a single frequency may occasionally yield a smaller error for a particular shape, such as $k=4$ for the pear-shaped scatterer, its performance is sensitive to the chosen wavenumber. In the same example, the error increases significantly when $k$ is changed to $5$. By contrast, the multi-frequency strategy provides more consistent reconstructions across different geometries and captures the spatial correlation structure of the random shape fluctuation with higher overall reliability.

\section{Concluding remarks}
\label{Conclusion}

In this work, we investigated an inverse acoustic scattering problem for three-dimensional random obstacles with uncertain shapes, extending the study of random obstacle scattering from two-dimensional settings to three dimensions. The uncertain scatterer was represented by a smooth radial function with small-amplitude random fluctuations, which allows both Gaussian and non-Gaussian random shape variations to be incorporated into a unified framework. With sufficiently many incident waves at a fixed frequency, we show that the distribution of the random radial function is uniquely determined by the distribution of the corresponding far-field data. Under the zero-mean perturbation model, this further implies the uniqueness of the deterministic reference shape and the associated second-order statistical quantities. On the computational side, a Monte Carlo-based multi-frequency recursive linearization method was developed to recover the samples of random scatterer realizations from far-field measurements: the nonlinear far-field mapping  is linearized with respect to the radial shape parameters at each iteration and the reconstruction is updated  along an increasing sequence of frequencies. Our reconstruction is not limited to deterministic shape recovery: the recovered samples were further used to estimate the mean geometry and key statistical features of the random perturbation field, including the covariance structure, KL eigenvalues, covariance hyper-parameters and representative marginal distributions. Numerical experiments for Gaussian and non-Gaussian perturbations demonstrate the effectiveness of the proposed method in recovering both geometric and statistical information of three-dimensional random scatterers from far-field data. To the best of our knowledge, this work is among one of the first to achieve quantitative numerical recovery of the statistical structure of three-dimensional random shape perturbations from far-field data.

For future work, a detailed analysis of the convergence and stability properties of the proposed recursive linearization algorithm in the three-dimensional random setting would be of significant interest. It would also be valuable to quantify the influence of sample size, noise level and frequency selection on the recovery of statistical quantities. From a theoretical perspective, extending the present uniqueness analysis to settings where only partial statistical information of the far-field data is available such as moment- or covariance-based data, is an interesting direction. Extensions to more general random shape models and more complex spatial correlation structures are also worth exploring, as they would further enhance the applicability of the proposed approach to inverse random  scattering problems with richer uncertainty structures.

\section*{Acknowledgments}
 We would like to thank the Center for High Performance Computing at Shanghai Jiao Tong University for the provision of Siyuan-1 cluster which supported the computational work conducted in this study.

 \section*{Data and Code Availability}
Both the research data and the source code used to conduct the numerical experiments and verify the conclusions of this study can be obtained from the corresponding author upon reasonable and justified request.



\bibliographystyle{unsrt}
\bibliography{IP_2_JCP}

\begin{thebibliography}{10}

\bibitem{colton1998inverse}
David Colton and Rainer Kress.
\newblock {\em Inverse Acoustic and Electromagnetic Scattering Theory}, volume~93 of {\em Applied Mathematical Sciences}.
\newblock Springer, Cham, Switzerland, 4th edition, 2019.

\bibitem{Bao2001MathematicalModeling}
G.~Bao, L.~Cowsar, and W.~Masters, editors.
\newblock {\em Mathematical Modeling in Optical Science}, volume~22 of {\em Frontiers in Applied Mathematics}.
\newblock SIAM, Philadelphia, PA, 2001.

\bibitem{Kuchment2014}
P.~Kuchment.
\newblock {\em The Radon Transform and Medical Imaging}, volume~85 of {\em CBMS-NSF Regional Conference Series in Applied Mathematics}.
\newblock SIAM, Philadelphia, PA, 2014.

\bibitem{wang2020review}
Junying Wang and Xinqian Zheng.
\newblock Review of geometric uncertainty quantification in gas turbines.
\newblock {\em J. Eng. Gas Turbines Power}, 142(7):070801, 2020.

\bibitem{MR3120587}
Gang Bao, Shui-Nee Chow, Peijun Li, and Haomin Zhou.
\newblock An inverse random source problem for the {H}elmholtz equation.
\newblock {\em Math. Comp.}, 83(285):215--233, 2014.

\bibitem{MR3565588}
Gang Bao, Chuchu Chen, and Peijun Li.
\newblock Inverse random source scattering problems in several dimensions.
\newblock {\em SIAM/ASA J. Uncertain. Quantif.}, 4(1):1263--1287, 2016.

\bibitem{MR4653392}
Peijun Li, Ying Liang, and Yuliang Wang.
\newblock A data-assisted two-stage method for the inverse random source problem.
\newblock {\em SIAM J. Imaging Sci.}, 16(4):1929--1952, 2023.

\bibitem{MR2386724}
Matti Lassas, Lassi P\"aiv\"arinta, and Eero Saksman.
\newblock Inverse scattering problem for a two dimensional random potential.
\newblock {\em Commun. Math. Phys.}, 279(3):669--703, 2008.

\bibitem{MR4712402}
Peijun Li and Xu~Wang.
\newblock Inverse scattering for the biharmonic wave equation with a random potential.
\newblock {\em SIAM J. Math. Anal.}, 56(2):1959--1995, 2024.

\bibitem{MR4918624}
Jianliang Li, Peijun Li, Xu~Wang, and Guanlin Yang.
\newblock Inverse random potential scattering for the polyharmonic wave equation using far-field patterns.
\newblock {\em SIAM J. Appl. Math.}, 85(3):1237--1260, 2025.

\bibitem{MR4164073}
Gang Bao, Yiwen Lin, and Xiang Xu.
\newblock Inverse scattering by a random periodic structure.
\newblock {\em SIAM J. Numer. Anal.}, 58(5):2934--2952, 2020.

\bibitem{MR4693214}
Hao Gu, Xiang Xu, and Liang Yan.
\newblock Inverse elastic scattering by random periodic structures.
\newblock {\em J. Comput. Phys.}, 501:112785, 2024.

\bibitem{MR5072888}
Yi~Wang, Lei Lin, and Junliang Lv.
\newblock Numerical method for the inverse scattering problem of acoustic-elastic interaction by random periodic structures.
\newblock {\em J. Comput. Phys.}, 562:115028, 2026.

\bibitem{sun2026inverse}
Zhiqi Sun, Xiang Xu, and Yiwen Lin.
\newblock Inverse acoustic scattering for random obstacles with multi-frequency data.
\newblock arXiv:2601.22560, 2026.

\bibitem{MR4337758}
J{\"u}rgen D{\"o}lz, Helmut Harbrecht, Carlos Jerez-Hanckes, and Michael Multerer.
\newblock Isogeometric multilevel quadrature for forward and inverse random acoustic scattering.
\newblock {\em Comput. Methods Appl. Mech. Eng.}, 388:114242, 2022.

\bibitem{MR883771}
David Colton and Peter Monk.
\newblock The numerical solution of the three-dimensional inverse scattering problem for time harmonic acoustic waves.
\newblock {\em SIAM J. Sci. Stat. Comput.}, 8(3):278--291, 1987.

\bibitem{MR1181580}
Rainer Kress and Axel Zinn.
\newblock On the numerical solution of the three-dimensional inverse obstacle scattering problem.
\newblock {\em J. Comput. Appl. Math.}, 42(1):49--61, 1992.

\bibitem{MR2363787}
Helmut Harbrecht and Thorsten Hohage.
\newblock Fast methods for three-dimensional inverse obstacle scattering problems.
\newblock {\em J. Integral Equations Appl.}, 19(3):237--260, 2007.

\bibitem{MR4929108}
Yunwen Yin and Liang Yan.
\newblock {TDDM}: A transfer learning framework for physics-guided {3D} acoustic scattering inversion.
\newblock {\em J. Comput. Phys.}, 539:114211, 2025.

\bibitem{MR729385}
David Colton and B.~D. Sleeman.
\newblock Uniqueness theorems for the inverse problem of acoustic scattering.
\newblock {\em IMA J. Appl. Math.}, 31(3):253--259, 1983.

\bibitem{MR4142771}
Paul Escapil-Inchausp\'e and Carlos Jerez-Hanckes.
\newblock Helmholtz scattering by random domains: first-order sparse boundary element approximation.
\newblock {\em SIAM J. Sci. Comput.}, 42(5):A2561--A2592, 2020.

\bibitem{Hadamard1923}
J.~Hadamard.
\newblock {\em Lectures on {Cauchy's} Problem in Linear Partial Differential Equations}.
\newblock Yale University Press, New Haven, 1923.

\bibitem{MR3404631}
Annika Lang and Christoph Schwab.
\newblock Isotropic {G}aussian random fields on the sphere: regularity, fast simulation and stochastic partial differential equations.
\newblock {\em Ann. Appl. Probab.}, 25(6):3047--3094, 2015.

\bibitem{marinucci2011random}
Domenico Marinucci and Giovanni Peccati.
\newblock {\em Random Fields on the Sphere: Representation, Limit Theorems and Cosmological Applications}, volume 389 of {\em London Mathematical Society Lecture Note Series}.
\newblock Cambridge University Press, Cambridge, 2011.

\bibitem{oksendal2013stochastic}
Bernt {\O}ksendal.
\newblock {\em Stochastic Differential Equations: An Introduction with Applications}.
\newblock Universitext. Springer, Heidelberg, 6th edition, 2003.

\bibitem{bao2015inverse}
Gang Bao, Peijun Li, Junshan Lin, and Faouzi Triki.
\newblock Inverse scattering problems with multi-frequencies.
\newblock {\em Inverse Probl.}, 31(9):093001, 2015.

\bibitem{MR4162000}
Carlos Borges and Jun Lai.
\newblock Inverse scattering reconstruction of a three-dimensional sound-soft axis-symmetric impenetrable object.
\newblock {\em Inverse Probl.}, 36(10):105005, 2020.

\bibitem{MR2998714}
Mourad Sini and Nguyen~Trung Th\`anh.
\newblock Inverse acoustic obstacle scattering problems using multifrequency measurements.
\newblock {\em Inverse Probl. Imaging}, 6(4):749--773, 2012.

\bibitem{MR1203018}
A.~Kirsch.
\newblock The domain derivative and two applications in inverse scattering theory.
\newblock {\em Inverse Probl.}, 9(1):81--96, 1993.

\bibitem{wkeglarczyk2018kernel}
Stanis{\l}aw W{\k{e}}glarczyk.
\newblock Kernel density estimation and its application.
\newblock {\em ITM Web Conf.}, 23:00037, 2018.

\bibitem{Kechris1995}
Alexander~S. Kechris.
\newblock {\em Classical Descriptive Set Theory}, volume 156 of {\em Graduate Texts in Mathematics}.
\newblock Springer, New York, 1995.

\bibitem{dai2013approximation}
Feng Dai and Yuan Xu.
\newblock {\em Approximation Theory and Harmonic Analysis on Spheres and Balls}.
\newblock Springer Monographs in Mathematics. Springer, New York, 2013.

\bibitem{MR3933278}
Ana Carpio, Thomas~G. Dimiduk, Fr\'ed\'erique Le~Lou\"er, and Mar\'ia~Luisa Rap\'un.
\newblock When topological derivatives met regularized {G}auss-{N}ewton iterations in holographic 3{D} imaging.
\newblock {\em J. Comput. Phys.}, 388:224--251, 2019.

\bibitem{wu1988non}
S.~C. Wu, M.~F. Chen, and A.~K. Fung.
\newblock {Non-Gaussian} surface generation.
\newblock {\em IEEE Trans. Geosci. Remote Sens.}, 26(6):885--888, 1988.

\end{thebibliography}

\end{document}